\documentclass{article}

\usepackage{amsmath,amsfonts,amssymb,url,amsthm,tocloft,enumitem}

\usepackage[linesnumbered,lined,boxed,commentsnumbered]{algorithm2e}
\usepackage{hyperref} %% For navigating the text
\usepackage[utf8]{inputenc}

\usepackage{color}

{%\theorembodyfont{\slshape}
\theoremstyle{plain}
\newtheorem{proposition}{Proposition}[section]
\newtheorem{lemma}[proposition]{Lemma}
\newtheorem{corollary}[proposition]{Corollary}

\newtheorem{theorem}[proposition]{Theorem}}

{%\theorembodyfont{\upshape}
\theoremstyle{definition}

}

\newcommand{\eps}{\varepsilon}

\newcommand\C{{\mathbb C}}

\newcommand\Q{{\mathbb Q}}

\newcommand\Z{{\mathbb Z}}

\newcommand{\CC}{{\mathcal{C}}}
\newcommand{\FF}{{\mathcal{F}}}

\newcommand{\OO}{{\mathcal{O}}}

\newcommand{\tilf}{{\tilde f}}

\renewcommand{\Im}{{\mathrm{Im}\,}}
\renewcommand{\Re}{{\mathrm{Re}\,}}
\newcommand{\height}{{\mathrm{h}}}

\newcommand{\ord}{{\mathrm{ord}}}

\title{No singular modulus is a unit}
\author{Yu. Bilu, P. Habegger and L. Kühne}

1
\numberwithin{equation}{section}

\begin{document}

\hfuzz 4pt

\date{}
\maketitle

\begin{flushright}
\textit{To David Masser}
\end{flushright}

\abstract{A result of the second-named author states that there are only finitely many CM-elliptic curves over $\mathbb{C}$ whose $j$-invariant is an algebraic unit. His proof depends on Duke's Equidistribution Theorem and is hence non-effective. In this article, we give a completely effective proof of this result. To be precise, we show that every singular modulus that is an algebraic unit is associated with a CM-elliptic curve whose endomorphism ring has discriminant less than $10^{15}$.  Through further refinements and  computer-assisted arguments, we eventually rule out all remaining cases, showing that no singular modulus is an algebraic unit. This allows us to exhibit classes of subvarieties in $\C^n$  not containing any special points.}

{\footnotesize

\tableofcontents

}

\section{Introduction}
\label{sintro}

Since the nineteenth century, $j$-invariants associated with elliptic curves having complex multiplication (CM), the so-called singular moduli, have been an object of study in number theory. A theorem of Weber \cite[Theorem~11.1]{Cox} states that every singular modulus is an algebraic integer. Under certain technical restrictions, Gross and Zagier \cite{Gross1985} stated explicit formulas for the absolute norm of the difference between two singular moduli.

Motivated by effective results of Andr\'e-Oort type \cite{BMZ13,
  Ku12}, David Masser raised in 2011 the question whether only
finitely many singular moduli are algebraic units, that is, units of the
ring of all algebraic integers. Throughout this article, we call such
hypothetical algebraic numbers \textsl{singular units}. Since there
is no example of a singular unit in the literature, it seems
legitimate to ask whether there are any singular units at all.

In~\cite{hab:junit}, the second-named author answered Masser's original  question in the affirmative: There exist at most finitely many singular units. However, his proof is non-effective as it invokes Siegel's lower bounds on the class number of imaginary quadratic fields \cite{Siegel1935} through Duke's Equidistribution Theorem \cite{Duke:hyperbolic}.

Here, we can give the following definite answer to Masser's question as our main theorem.

\begin{theorem}
\label{thmain}
There are no singular units.
\end{theorem}

%As already mentioned in \cite{hab:junit}, the explicit formulas of Gross and Zagier \cite{Gross1985}, which were extended by Dorman \cite{Dorman1988}, can be used to compute the absolute norm of many singular moduli. However, it seems hard to extract a result like the above theorem from these formulas.

Theorem~\ref{thmain} is a formal consequence of our Theorems~\ref{thhighrange},~\ref{thmidrange},~\ref{thlowrange} and~\ref{thm:habeggerspari}. Let us briefly 
sketch its proof.
We say that a
singular modulus is of discriminant~$\Delta$ if it is the
$j$-invariant of a CM-elliptic curve whose endomorphism ring is the
imaginary quadratic order of discriminant~$\Delta$. We also write
${\Delta=Df^2}$ where~$D$ is the discriminant of the CM-field
$\Q(\sqrt\Delta)$, the fundamental discriminant, and $f$ is the
conductor of the endomorphism order. The singular moduli of a given
discriminant $\Delta$ form a full Galois orbit over~$\Q$ of
cardinality equal to the class number\footnote{We do not use the more traditional notation $h(\Delta)$ because of the risk of confusing it with the height $\height(\cdot)$.} $\CC(\Delta)$. 

Write $\zeta_3$ (resp.\ $\zeta_6$) for the third (resp.\ sixth) root
of unity $e^{2\pi i/3}$ (resp.\ $e^{\pi i/3}$). Note that $\zeta_3$
(resp.\ $\zeta_6$) is the left (resp.\ right) vertex of the geodesic
triangle enclosing the standard fundamental domain $\mathcal{F}$ in
the Poincaré upper half-plane. Given ${\eps\in (0,1/3]}$, denote by
  $\CC_\eps(\Delta)$ the number of singular moduli of
  discriminant~$\Delta$ which can be written ${j(\tau)}$ where~$\tau
  \in \mathcal{F}$ satisfies ${|\tau-\zeta_3|< \eps}$ or
  ${|\tau-\zeta_6|< \eps}$ and ${j(\cdot)}$ denotes Klein's $j$-function. Since $\zeta_3$ and $\zeta_6$ are the only
  zeros  of the $j$-function contained in the closure of
  $\mathcal{F}$, a pivotal ingredient in the proof of Theorem
  \ref{thmain} is an upper bound on $\CC_\eps(\Delta)$. Indeed, a main
  point of the argument in \cite{hab:junit} is the
  estimate
  $${\CC_\eps(\Delta)\ll \CC(\Delta)\eps^2}$$
which holds when~$|\Delta|$ is sufficiently large (in terms of~$\eps$). Unfortunately, ``sufficiently large'' here
   is not effective; in fact, this is the place  where 
 Duke's Equidistribution Theorem~\cite{Duke:hyperbolic}, generalized by Clozel and Ullmo~\cite{CU:Equidistribution} to arbitrary discriminants, is used.

Our main novelty is the following effective estimate (see Theorem~\ref{thceps}):
\begin{equation}
\label{ecepsrough}
\CC_\eps(\Delta)  \ll F\left(\frac{\sigma_1(f)}{f}|\Delta|^{1/2}\eps^2+|\Delta|^{1/2}\eps+ \sigma_0(f)  |
\Delta|^{1/4}\eps+1\right),
\end{equation}
where
$$
F=F(\Delta) = \max \bigl\{2^{\omega(a)}: a\le {|\Delta|^{1/2}}\bigr\}.
$$
Here  and in the sequel all implicit constants are effective,  and we use the standard notation
\begin{align}
  \label{eq:standardnotation}
\omega(n)=\sum_{p\mid n}1,\quad \sigma_0(n)=\sum_{d\mid n}1,\quad \sigma_1(n)=\sum_{d\mid n}d. 
\end{align}
Using that
\begin{equation}
\label{eolog}
\omega(n)=o(\log n),\quad \log\sigma_0(n)=o(\log n), \quad \sigma_1(n)\ll n\log\log n,
\end{equation}
we deduce from~\eqref{ecepsrough} that%, for any ${\eta>0}$ we have  
\begin{equation}
\label{ecepseta}
\CC_\eps(\Delta)  \le |\Delta|^{1/2+o(1)}\eps+  |
\Delta|^{o(1)}
\end{equation}
%where ${A=A(\Delta)=|\Delta|^{o(1)}}$ 
as ${|\Delta|\to\infty}$. 

The height of a singular unit~$\alpha$ of discriminant~$\Delta$ can be easily estimated in terms of $\CC_\eps(\Delta)$:
\begin{equation}
\label{eupperheight}
\height(\alpha)\ll\frac{\CC_\eps(\Delta)}{\CC(\Delta)}\log|\Delta|+\log(\eps^{-1}),
\end{equation}
see Theorem~\ref{theceps}.
(By the \textit{height} we mean here the usual \textit{absolute logarithmic height} of an algebraic number; its definition is recalled in the beginning of Section~\ref{supper}.) Substituting \eqref{ecepseta}, we obtain the following upper estimate:
\begin{equation*}
%\label{eheta}
\height(\alpha)\le \frac{|\Delta|^{1/2}}{\CC(\Delta)}A\eps+  \frac{|
\Delta|^{o(1)}}{\CC(\Delta)}+O(\log(\eps^{-1}))
\end{equation*}
where ${A=A(\Delta)=|\Delta|^{o(1)}}$ 
as ${|\Delta|\to\infty}$. Specifying 
$$
\eps= \frac{\CC(\Delta)}{A|\Delta|^{1/2}}
$$
(which is a nearly optimal choice), we obtain the estimate 
$$
\height(\alpha)\le   \frac{|
\Delta|^{o(1)}}{\CC(\Delta)}+o(\log|\Delta|) +O\left(\log^+\frac{|\Delta|^{1/2}}{\CC(\Delta)}\right)
$$
where ${\log^+(x)=\max\{\log x,0\}}$. 

To obtain an upper bound on $|\Delta|$, we combine this bound with the following two lower estimates on $\height(\alpha)$ (see Section \ref{slower})
\begin{align}
\label{ecolm}
\height(\alpha)+1&\gg \log|\Delta|, \\
\label{etriv}
\height(\alpha)&\gg \frac{|\Delta|^{1/2}}{\CC(\Delta)}.
\end{align}
The bound~\eqref{ecolm} is rather deep and relies on work of Colmez~\cite{Colmez} and Nakkajima-Taguchi \cite{NT}. On
the contrary,~\eqref{etriv}  follows easily from the fact that one of
the conjugates of our singular unit~$\alpha$ is
${j\bigl((\Delta+\sqrt\Delta)/2\bigr)}$. Nevertheless, \eqref{etriv} plays a
crucial role when the class number is pathologically small so that it
would contradict the Generalized Riemann Hypothesis (GRH). In fact, \eqref{etriv} becomes much stronger than \eqref{ecolm} in these hypothetical cases.

Comparing upper and lower estimates, we obtain for large $|\Delta|$ that
$$
\max\left\{\frac{|\Delta|^{1/2}}{\CC(\Delta)}, \log|\Delta|\right\}\le \frac{|
\Delta|^{o(1)}}{\CC(\Delta)}+o(\log|\Delta|) +O\left(\log^+\frac{|\Delta|^{1/2}}{\CC(\Delta)}\right),
$$
which is clearly impossible.

To get an explicit bound on $|\Delta|$, we need to replace all implicit constants above with explicit ones. This relies in particular on a numerically sharp estimate for the arithmetic function $\omega(n)$ due to Robin~\cite{Ro83}. In Section \ref{stenfourteen}, we see that this leads to a bound $|\Delta|< 10^{15}$. While already  effective, it is still not feasible to check directly by a computer-assisted proof that none of the singular moduli of discriminant ${\Delta \in (- 10^{15}, -3]}$ is an algebraic unit.

A refinement of our original arguments comes to our rescue. When  ${|\Delta|<10^{15}}$, we improve on the estimate~(\ref{ecepsrough}) by bounding sums of the form $\sum_{n \in [a,b] \cap \Z} 2^{\omega(n)}$ in a more refined way. A natural idea is to use the Selberg-Delange method, which yields the asymptotic expansion
\begin{equation}
\label{equation::selbergdelange}
S(x)= \sum_{n\le x} 2^{\omega(n)} = x \log x\left(\lambda_0 + \frac{\lambda_1}{\log x} + O\bigl(e^{-c\sqrt{\log x}}\bigr)\right)
\end{equation} 
with explicit constants $\lambda_0,\lambda_1 \in \mathbb{R}$ and some
 constant ${c>0}$ (see, for instance,  \cite[Theorem
  II.6.1]{Te15}). There are two downsides of this method. First, the
error term is suboptimal under assumption of the GRH. Second, it would need some effort to make the constant $c$ actually explicit. 

However, as $|\Delta|<10^{15}$, we are only interested in the case
where $[a,b]$ is a subinterval of  ${[1, 2 \cdot 10^7]}$. In this
range, a simple \textsf{SAGE} script using the \textsf{MPFI} library
\cite{MPFI,sagemath} can be used to improve on
\eqref{equation::selbergdelange}  computationally (see Proposition
\ref{proposition::selbergdelange}). As a consequence, we obtain
$|\Delta|<10^{10}$ for any singular unit of discriminant $\Delta$ in Theorem \ref{thmidrange}.

This is still not sufficient to check all remaining cases, at least
with modest computational means. The range is nevertheless small
enough to use a counting algorithm in order to bound
$\CC_{10^{-3}}(\Delta)$ for all discriminants $\Delta$ satisfying
$|\Delta|< 10^{10}$, see Lemma \ref{lem:lowrangeCebound}. {This still
needs an appropriate counting strategy, as determining
$\CC_{10^{-3}}(\Delta)$ for each discriminant is rather slow,
comparable to computing separately each class number $\CC(\Delta)$ in the
same range. Our trick is to bound all $\CC_{10^{-3}}(\Delta)$
simultaneously by running through a set containing all imaginary quadratic
$\tau \in \mathcal{F}$ satisfying ${|\tau-\zeta_3|< \eps}$ or
${|\tau-\zeta_6|< \eps}$ and such that $j(\tau)$ is of discriminant
$\Delta$ with $|\Delta|<10^{10}$. For each $\tau$ encountered, we
compute its discriminant $\Delta(\tau)$ after the fact and increment
our counter for $\CC_{10^{-3}}(\Delta(\tau))$. The thus obtained
bounds for $\CC_{10^{-3}}(\Delta)$ refine once again our previous
inequalities, and allow us to conclude that $|\Delta|<10^7$. Repeating this
procedure once again, with a slightly changed $\varepsilon$, we
achieve even $|\Delta|<3 \cdot 10^5$ in Theorem \ref{thlowrange}.
These remaining cases can now be dealt with directly, for which we use
a \textsf{PARI}~\cite{PARI} program to prove Theorem  \ref{thm:habeggerspari}, completing thereby the proof of Theorem~\ref{thmain}. 

\bigskip

It is very probable that our argument can be adapted to solve a more general problem: given an algebraic integer~$\beta$,  determine the singular moduli~$\alpha$ such that ${\alpha-\beta}$ is a unit; or at least bound effectively the discriminants of such~$\alpha$. For instance, one may ask whether~$0$ is the only singular modulus~$\alpha$ such that ${\alpha-1}$ is a unit. In the general case, as explained in~\cite{hab:junit}, this would require lower bounds for elliptic logarithmic forms, but when~$\beta$ itself is a singular modulus, our argument extends almost without changes. One may go further and obtain an effective version of Theorem~2 from~\cite{hab:junit}, which is an analogue of Siegel's Finiteness  Theorem for
special points.

%an effective version of the ``Siegel Theorem for special points'', see Theorem~2 in~\cite{hab:junit}.

The famous work of Gross-Zagier and Dorman \cite{Dorman1988,Gross1985} inspires the following problem: determine all couples $(\alpha,\beta)$ of singular moduli such that ${\alpha-\beta}$ is a unit; presumably, there is none. As indicated above, when~$\beta$ is fixed and~$\alpha$ varying, a version of our argument does the job, but if we let both~$\alpha$ and~$\beta$ vary, the problem seems more intricate. Very recently Yingkun Li~\cite{Li18} made important progress: he proved that ${\alpha-\beta}$ is not a unit if the discriminants of~$\alpha$ and~$\beta$ are  fundamental and coprime. In particular, his result implies the following partial version of our Theorem~\ref{thmain}: the discriminant of a singular unit must be either non-fundamental or divisible by~$3$. 

Another natural problem is extending our work to $S$-units. Recall that, given a finite set~$S$ of prime numbers, a non-zero algebraic number is called an $S$-unit if both its denominator and numerator are composed of prime ideals dividing primes from~$S$.  Recently %Sebastián  
Herrero, %Ricardo 
Menares and %Juan 
Rivera-Letelier announced the proof of finiteness of the set of singular $S$-units (that is, singular moduli that are $S$-units) for any finite set of primes~$S$. However, to the best of our knowledge, their argument is not effective as of now. 

% and tackling it would probably require substantially new ideas.

\bigskip

Finally, let us discuss an application of Theorem~\ref{thmain} to
effective results of Andr\'e-Oort type. A point
$(\alpha_1,\ldots,\alpha_n) \in \mathbb{C}^n$ is called special if
each $\alpha_i$, $i \in \{ 1, \dots, n\}$, is a  singular modulus. Since singular moduli are algebraic integers, the following statement is an immediate consequence of our main result.

%% Let $\mathcal{O}_{\overline{\mathbb{Q}}}$ be the ring of algebraic integers.

\begin{corollary} For each polynomial ${P}$ in unknowns $X_2,\ldots,
  X_n$ and  coefficients that are algebraic integers in~$\C$, %%  \in \mathcal{O}_{\overline{\mathbb{Q}}}[X_2,\dots,X_n]}$ %and each ${\beta \in \mathcal{O}_{\overline{\mathbb{Q}}}^\times}$, 
  %% the hypersurface
  the hypersurface defined by 
\begin{equation*}
%\label{ehypersurface}
X_1 P(X_1, \dots, X_n) = 1 %\beta
\end{equation*}
 contains no special points.
\end{corollary}

In particular, $\alpha_1^{a_1}\cdots \alpha_n^{a_n}\not=1$
for all special points $(\alpha_1,\ldots,\alpha_n)$
and all integers $a_1\ge 1,\ldots,a_n \ge 1$. 
This corollary exhibits a rather general class of algebraic varieties
of arbitrary dimension and degree for which  the celebrated theorem of
Pila~\cite{Pila2011a} can be proved effectively and even explicitly.
It is complementary to other recent effective results of Andr\'e-Oort type~\cite{Bilu2017,Bi18}. %recently obtained by the first- and third-named author

\paragraph{Plan of the article}
In Section~\ref{sceps} we obtain an explicit version of the
estimate~\eqref{ecepsrough}. In Section~\ref{supper} we obtain an
upper estimate  for the height of a singular unit. In
Section~\ref{slower} we obtain explicit versions of  the lower estimates~\eqref{ecolm} and~\eqref{etriv}. In Section~\ref{stenfourteen} we use all  previous results to bound the discriminant of a singular unit as  ${|\Delta|<10^{15}}$. This bound is reduced to $10^{10}$ in Section~\ref{sec:midrange} and to $3\cdot10^5$ in Section~\ref{slowrange}.  Finally, in Section~\ref{sfinal} we show that the discriminant of a singular unit satisfies ${|\Delta|>3\cdot10^5}$. 

\paragraph{Convention}
In this article we fix, once and for all, an embedding ${\bar\Q\hookrightarrow \C}$; this means that all algebraic numbers in this article are viewed as elements of~$\C$.

\paragraph{Acknowledgments}
Yuri Bilu was partially supported by the University of Basel, the Fields Institute (Toronto), and the Xiamen University. Lars Kühne was supported by the Max-Planck Institute for Mathematics, the Fields Institute, and the Swiss National Science Foundation through an Ambizione grant. We thank Ricardo Menares and Amalia Pizarro for many useful conversations, Florian Luca and Aleksandar Ivic for helpful suggestions, Bill Allombert and Karim Belabas for a \textsf{PARI} tutorial, and Jean-Louis Nicolas and Cyril Mauvillain for helping to access Robin's thesis~\cite{Ro83a}.  Finally, we thank both anonymous referees for encouraging reports and helpful suggestions.

\section{An estimate for \texorpdfstring{$\CC_\eps(\Delta)$}{Ceps(Delta)}}
\label{sceps}
Let~$\Delta$ be a negative integer satisfying ${\Delta\equiv 0,1\bmod 4}$ and  
$$
\OO_\Delta=\Z[(\Delta+\sqrt\Delta)/2] 
$$
the  imaginary quadratic order of discriminant~$\Delta$. Then ${\Delta=Df^2}$, where~$D$ is the discriminant of the imaginary quadratic field ${\Q(\sqrt\Delta)}$  (the ``fundamental discriminant'') and ${f=[\OO_D:\OO_\Delta]}$ is the conductor. We denote by $\CC(\Delta)$ the class number  of the order~$\OO_\Delta$.

Up to $\C$-isomorphism there exist $\CC(\Delta)$ elliptic curves with CM by~$\OO_\Delta$. The $j$-invariants of these curves are called \textsl{singular moduli} of discriminant~$\Delta$. %The Class Field Theory implies that the $\CC(\Delta)$  
The singular moduli  of discriminant~$\Delta$ form a full Galois orbit over~$\Q$
of cardinality $\CC(\Delta)$, see \cite[Proposition~13.2]{Cox}.

Let~$\FF$ be the standard fundamental domain in the Poincaré  plane,  that is, the open hyperbolic triangle with vertices ${\zeta_3,\zeta_6,i\infty}$, together with the geodesics $[i,\zeta_6]$ and ${[\zeta_6,i\infty)}$; here 
$$
\zeta_3=e^{2\pi i/3}=\frac{-1+\sqrt{-3}}2, \quad \zeta_6=e^{\pi i/3}=\frac{1+\sqrt{-3}}2. 
$$
Every singular modulus can be uniquely presented as  ${j(\tau)}$, where ${\tau\in \FF}$.

Now fix ${\eps\in (0,1/3]}$ and denote by $\CC_\eps(\Delta)$ the number of singular moduli of discriminant~$\Delta$ that can be presented as ${j(\tau)}$ where ${\tau\in \FF}$ satisfies %${\Im\tau\le \frac{\sqrt3}{2}+\eps}$. 
\begin{equation}
\label{etauzeze}
\min \{|\tau-\zeta_3|,|\tau-\zeta_6|\}< \eps. 
\end{equation}
In this section we bound this quantity.

Define the \textsl{modified conductor}~$\tilf$ by
\begin{equation}
\label{etilf}
\tilf=
\begin{cases}
f,& D\equiv 1\bmod 4,\\
2f, &D\equiv 0\bmod 4. 
\end{cases}
\end{equation}
Then ${\Delta/\tilf^2}$ is a square-free integer.

\begin{theorem}
\label{thceps}
For ${\eps\in (0,1/3]}$ we have
\begin{equation}
\label{enewbound}
\CC_\eps(\Delta)  \le  F\left( \frac{16}3\frac{\sigma_1(\tilf)}{\tilf}|\Delta|^{1/2}\eps^2+\frac83|\Delta|^{1/2}\eps+ 8|\Delta/3|^{1/4}\sigma_0(\tilf)\eps+4\right),
\end{equation}
where  
\begin{equation}
\label{ecapitalf}
F=F(\Delta) = \max \bigl\{2^{\omega(a)}: a\le |\Delta|^{1/2}\bigr\}.
\end{equation}
\end{theorem}
\begin{corollary}
\label{cseps}
In the set-up of Theorem~\ref{thceps} assume that ${|\Delta|\ge 10^{14}}$.  Then  
\begin{equation}
\label{enewboundsimple}
\CC_\eps(\Delta)  \le  F\left( 9.83|\Delta|^{1/2} \eps^2\log\log(|\Delta|^{1/2})+3.605|\Delta|^{1/2}\eps+ 4\right).
\end{equation}
\end{corollary}

\subsection{Some lemmas}
We need some lemmas. For a prime number~$\ell$ and a non-zero integer~$n$ we denote by $\ord_\ell(n)$ the $\ell$-adic order of~$n$; that is, ${\ell^{\ord_\ell(n)}\,\|\,n}$. 

\begin{lemma}
\label{lrootsmodle}
Let~$\ell$ be a prime number, ${e \ge 1}$ an integer, and~$\Delta$ a
non-zero integer with ${\nu=\ord_\ell\Delta}$.
Then the set of ${b\in \Z}$ satisfying 
${b^2\equiv \Delta\bmod \ell^e}$ 
is a union of at most~$2$ residue classes modulo ${\ell^{e-\lfloor
    \min\{e,\nu\}/2\rfloor}}$ in all cases except when ${\ell=2}$ and
${e\ge 3}$; in this latter case it is a union of most~$4$ such
classes. Finally, the set of $b$ equals a single residue class modulo ${\ell^{e-\lfloor
    \min\{e,\nu\}/2\rfloor}}$
if $\nu\ge e$. 
%%
%% Let $S$ be the set of ${b\in \Z}$ satisfying 
%% ${b^2\equiv \Delta\bmod \ell^e}$.
%%
%% \begin{enumerate}
%% \item [(i)] If $\nu <e$ then $S$
%% is a union of at most~$2$ residue classes modulo ${\ell^{e-\lfloor
%%     \min\{e,\nu\}/2\rfloor}}$ in all cases except when ${\ell=2}$ and
%% ${e\ge 3}$; in this latter case it is a union of at most~$4$ such
%% classes.
%% \item[(ii)] If $\nu\ge e$, then $S$ 
%% is a residue classes modulo ${\ell^{e-\lfloor
%%     \min\{e,\nu\}/2\rfloor}}$.
%% \end{enumerate}
\end{lemma}
\begin{proof}
We suppose first that ${\nu=0}$, that is, ${\ell\nmid \Delta}$.  In this case we have to count the number of elements in the multiplicative group $(\Z/\ell^e\Z)^\times$ whose
square is represented by~$\Delta$. If ${\ell\ge 3}$ or ${\ell^e\in \{2,4\}}$, then $(\Z/\ell^e\Z)^\times$ is a cyclic group.
Then there are at most~$2$ square roots and this implies our claim. If ${\ell=2}$ and  ${e \ge 3}$, then ${(\Z/2^e\Z)^\times\cong \Z/2\Z\times \Z/2^{e-2}\Z}$, and there are at most~$4$ square roots, as desired.

Now assume that  ${\nu <e}$. Then ${\ord_\ell(b^2) = \nu}$. So~$\nu$ is even and we can write ${b = \ell^{\nu/2}b'}$, where ${b' \in \Z}$ is coprime to~$\ell$. Now ${\Delta=\ell^\nu\Delta'}$ 
with ${\Delta'\in \Z}$ coprime to~$\ell$,  and ${(b')^2 \equiv \Delta'\bmod \ell^{e-\nu}}$.  Above we already determined
that, depending on the value of $\ell^{e-\nu}$, the set of
possible~$b'$ consists of either at most~$2$ or at most~$4$ classes
modulo ${\ell^{e-\nu}}$. Hence the set of possible ${b =
  \ell^{\nu/2}b'}$ consists of the same number of classes modulo
${\ell^{e-\nu/2}}$, as desired.

To prove the final claim  assume that ${\nu \ge e}$. In this case
${b^2\equiv \Delta\bmod \ell^e}$ 
is equivalent to ${b\equiv0\bmod \ell^{\lceil e/2\rceil}}$. This means
that the set of suitable~$b$ consists of exactly one class  modulo
${\ell^{\lceil e/2\rceil}=\ell^{e-\lfloor e/2\rfloor}}$.
\end{proof}

We say that ${d\in \Z}$ is a \textsl{quadratic divisor} of ${n\in \Z}$ if ${d^2\mid n}$. We denote by ${\gcd_2(m,n)}$ the greatest common quadratic divisor of~$m$ and~$n$. 

\begin{lemma}
\label{lrootsmoda}
Let~$a$ be a positive integer and~$\Delta$ a non-zero integer. Then the set of 
${b\in \Z}$ satisfying 
${b^2 \equiv \Delta \bmod a}$
consists of at most   ${2^{\omega(a/\gcd(a,\Delta))+1}}$ residue classes modulo 
${a/\gcd_2(a,\Delta)}$.% Moreover, if ${8\nmid a}$ then it is contained in at most ${2^{\omega(a)}}$ classes modulo ${a/\gcd_2(a,\Delta)}$.
\end{lemma}
\begin{proof}
  For a prime power $\ell^e$ 
  we only need the following simple consequence of 
  Lemma~\ref{lrootsmodle} on the number of residue classes counted there.
This number is at
most $2^{\omega(\ell^e/\gcd(\ell^e,\Delta))}$ if $\ell\ge 3$
and at most $2^{\omega(\ell^e/\gcd(\ell^e,\Delta))+1}$ for $\ell=2$.
  The current lemma follows from the Chinese
Remainder Theorem.
%% Indeed, if $8\mid a$ then we
%% pick up an extra factor~$2$ at ${\ell=2}$, which leads to the ``$+1$''
%% in $2^{\omega(a)+1}$.
\end{proof}

The  following lemma is trivial, but we state it here because it is our principal counting tool.  

\begin{lemma}
\label{ltrivi}
Let~$\alpha$ and~$\beta$ be real numbers, ${\alpha<\beta}$, and~$m$ a positive integer. Then every residue class modulo~$m$ has at most ${(\beta-\alpha)/m+1}$ elements in the interval ${[\alpha,\beta]}$.
\end{lemma}

Given a negative integer ${\Delta\equiv0,1\bmod 4}$, denote by ${T=T_\Delta}$ the set of triples of integers $(a,b,c)$ such that 
\begin{equation}
\label{ekuh}
\begin{aligned}
&\gcd(a,b,c)=1, \quad \Delta=b^2-4ac,\\
&\text{either\quad $-a < b \le a < c$\quad or\quad $0 \le b \le a = c$.}
\end{aligned}
\end{equation}
For  ${(a,b,c)\in T_\Delta}$ we set
$$
\tau(a,b,c)=\frac{b+\sqrt{\Delta}}{2a}.
$$
\begin{lemma}
\label{lgauss}
\begin{enumerate}[label={(\roman*)}]
\item
\label{ifund}
For every ${(a,b,c)\in T_\Delta}$ the number 
$\tau(a,b,c)$
belongs to the standard fundamental domain. 

\item
\label{iadelta}
For ${(a,b,c)\in T_\Delta}$ we have ${0<a\le |\Delta/3|^{1/2}}$, the equality being possible only if ${\Delta=-3}$ (and ${a=b=c=1}$). We also have ${c\ge |\Delta|^{1/2}/2}$. 

\item
\label{icl}
The map
${(a,b,c)\mapsto j(\tau(a,b,c))}$
defines a bijection from $T_\Delta$ onto the set of $\Q$-conjugates
of $j(\tau)$. 
In particular, ${\CC(\Delta)=|T_\Delta|}$.
\end{enumerate}
\end{lemma}

\begin{proof}
%This is, essentially, due to Gauss. 
For item~\ref{ifund} just note that~\eqref{ekuh} implies the inequalities 
$$
-\frac12 <\frac{b}{2a}\le \frac12, \qquad \frac{b^2+|\Delta|}{4a^2}\ge 1
$$
and that the second one becomes equality only when ${a=c}$, in which case  ${b\ge 0}$. 
For item~\ref{iadelta}, 
since ${|b|\le a\le c}$,  we have 
$$
4c^2\ge |\Delta|=4ac-b^2\ge 4a^2-a^2=3a^2. 
$$
with equality on the right only when ${a=|b|=c}$. Since ${\gcd(a,b,c)=1}$, this is only possible when ${a=b=c=1}$ and ${\Delta=-3}$. 

Item~\ref{icl} is a combination of several classical results that can be found, for instance, in~\cite{Cox}. See \cite[Proposition~2.5]{BLP16} for more details. 
\end{proof}

\begin{lemma}
\label{lshort}
Let ${\eps\in (0,1/3]}$ and let ${(a,b,c)\in T_\Delta}$ satisfy~\eqref{etauzeze}. Then 
\begin{align}
\label{eboundinga}
\frac{|\Delta|^{1/2}}{\sqrt3+2\eps}&<\hphantom{|}a\hphantom{|}\le \frac{|\Delta|^{1/2}}{\sqrt3},\\
\label{eboundingb}
a(1-2\eps)&<|b|\le a,\\
\label{eboundingc}
a&\le \hphantom{|}c\hphantom{|}<a(1+{\sqrt3}\eps+\eps^2).%, \\
%\frac{|\Delta|^{1/2}}{\sqrt3+2\eps}&<\hphantom{|}c\hphantom{|}<\frac{|\Delta|^{1/2}}{\sqrt3}\left(1+\frac{\sqrt3}2\eps+\frac14\eps^2\right)^{1/2}.
\end{align}
\end{lemma}
Note that~\eqref{eboundinga} and~\eqref{eboundingb} will be used already in Subsection~\ref{ssproofthceps}, while~\eqref{eboundingc} will be used only in Section~\ref{slowrange}.

\begin{proof}
For ${\tau\in \FF}$ condition~\eqref{etauzeze} implies that 
$$
\frac{\sqrt3}2\le\Im\tau<\frac{\sqrt3}2+\eps, \qquad 
\frac12-\eps<|\Re\tau|\le \frac12.
$$
Applying this for ${\tau=\tau(a,b,c)}$, we obtain~\eqref{eboundinga} and~\eqref{eboundingb}. To prove~\eqref{eboundingc}, write
$$
4ac= |\Delta|+b^2 < a^2(\sqrt3+2\eps)^2+a^2=4a^2(1+\sqrt3\eps+\eps^2), 
$$ 
and~\eqref{eboundingc} follows. 
\end{proof}

\subsection{Proof of Theorem~\ref{thceps}}
\label{ssproofthceps}

%\begin{proof}[Proof of Theorem~\ref{thceps}]
Note that, by definition, 
$$
\CC_\eps(\Delta) = \#\{\text{$(a,b,c)\in T_\Delta$~: ${\tau=\tau(a,b,c)}$ satisfies~\eqref{etauzeze}}\}. 
$$
Setting 
$$
I=\left(\frac{|\Delta|^{1/2}}{\sqrt3+2\eps}, \frac{|\Delta|^{1/2}}{\sqrt3}\right], 
$$
for ${\tau=\tau(a,b,c)}$ with ${(a,b,c)\in T_\Delta}$ we may re-write~\eqref{eboundinga} and~\eqref{eboundingb} as   
\begin{align}
\label{eainibin}
a\in I,\qquad b\in [-a,-a(1-2\eps))\cup (a(1-2\eps),a]. 
\end{align}
Since~$c$ is uniquely determined for given~$a$,~$b$ and~$\Delta$, it suffices to bound the number of  pairs $(a,b)$ of integers satisfying ${b^2\equiv \Delta\bmod a}$ and~\eqref{eainibin}. 

For every fixed~$a$ there are at most
${(4\eps\gcd_2(a,\Delta)+2) 2^{\omega(a)+1}}$ suitable~$b$, as follows from Lemmas~\ref{lrootsmoda} and~\ref{ltrivi};
indeed, ${\omega(a/\gcd(a,\Delta))\le \omega(a)}$. Hence
\begin{align}
\label{estim0}
\CC_\eps(\Delta) &\le 8\eps\sum_{a\in I\cap\Z} {\gcd}_2(a,\Delta) 2^{\omega(a)} + 4\sum_{a\in I\cap\Z}2^{\omega(a)}\\ 
\label{estim1}
&\le  8\eps F\sum_{a\in I\cap\Z} {\gcd}_2(a,\Delta) +4F\#(I\cap\Z).
\end{align}
%where~$F''$ is defined in~\eqref{efsecond}. 
To estimate the sum, note that 
\begin{equation}
\label{estim2}
\sum_{a\in I\cap\Z} {\gcd}_2(a,\Delta)\le \sum_{d^2\mid \Delta}d\cdot \#(I\cap d^2\Z). 
\end{equation}
Recall that we defined in~\eqref{etilf} the modified conductor~$\tilf$. 
Since $\Delta/\tilf^2$  is a square-free integer, we have ${d^2\mid \Delta}$ if and only if ${d\mid \tilf}$. Also, since~$I$ is of length 
$$
|\Delta|^{1/2}\left(\frac1{\sqrt3}-\frac1{\sqrt3+2\eps}\right)<\frac23|\Delta|^{1/2}\eps,
$$ 
we have, by Lemma~\ref{ltrivi},
$$
\#(I\cap d^2\Z) \le 
\begin{cases}
\frac23\frac{|\Delta|^{1/2}}{d^2}\eps +1, & d \le {|\Delta/3|^{1/4}}, \\
0,& d >{|\Delta/3|^{1/4}}. 
\end{cases}
$$
Hence
\begin{align}
\sum_{d^2\mid \Delta}d\cdot \#(I\cap d^2\Z) &\le \sum_{\genfrac{}{}{0pt}{}{d\mid\tilf}{d\le|\Delta/3|^{1/4}}}d\left(\frac23\frac{|\Delta|^{1/2}}{d^2}\eps+1\right)\nonumber\\
&\le \frac23|\Delta|^{1/2}\eps\sum_{d\mid\tilf}d^{-1}+ \sum_{\genfrac{}{}{0pt}{}{d\mid\tilf}{d\le |\Delta/3|^{1/4}}}d\nonumber\\
\label{estim3}
&\le \frac23\frac{\sigma_1(\tilf)}{\tilf}|\Delta|^{1/2}\eps+ |\Delta/3|^{1/4}\sigma_0(\tilf). 
\end{align}
Finally, Lemma~\ref{ltrivi} implies that
\begin{equation}
\label{eintini}
\#(I\cap\Z)\le \frac23|\Delta|^{1/2}\eps+1. 
\end{equation}
Putting  the estimates~\eqref{estim1},~\eqref{estim2},~\eqref{estim3} and~\eqref{eintini} together, we obtain~\eqref{enewbound}. \qed
%with~$F''$ instead of~$F$. Since ${F''\le F'\le F}$, this proves the theorem. 
%\end{proof}

\subsection{Proof of Corollary~\ref{cseps}}

We need to estimate  $\sigma_0(\tilf)$ and $\sigma_1(\tilf)$ in terms of~$|\Delta|$.  The following lemma uses a simple estimate for $\sigma_0(n)$ due to Nicolas and Robin~\cite{NR83}. Much sharper  estimates can be found in Robin's thesis~\cite{Ro83a}. 

\begin{lemma}
\label{lsigzersimple}
For ${|\Delta|\ge 10^{14}}$ we have 
\begin{align}
\label{euppersigzer}
\sigma_0(\tilf)&\le |\Delta|^{0.192},\\
\label{euppersigone}
{\sigma_1(\tilf)}/\tilf &\le 1.842\log\log(|\Delta|^{1/2}). 
\end{align}
\end{lemma}

\begin{proof}
For proving~\eqref{euppersigzer}  may assume that ${\tilf\ge 16}$, otherwise there is nothing to prove. In~\cite{NR83} it is proved that for ${n\ge 3}$ we have
$$
\frac{\log\sigma_0(n)}{\log2} \le 1.538\frac{\log n}{\log\log n}.
$$
The function ${x\mapsto (\log x)/(\log\log x)}$ is increasing for ${x\ge 16}$. Since 
$$
|\Delta|\ge 10^{14}, \qquad 16\le \tilf\le |\Delta|^{1/2},
$$ 
this gives
\begin{align*}
\log\sigma_0(\tilf) 
&\le 1.538\log 2\frac{\log(|\Delta|^{1/2})}{\log\log(|\Delta|^{1/2})}\\
&\le \frac{1.538}2\log 2\frac{\log|\Delta|}{\log\log(10^{7})}\\
&<0.192\log|\Delta|,
\end{align*}
as wanted. 

For proving~\eqref{euppersigone} we use the estimate 
${\sigma_1(n)\le 1.842n\log\log n}$
which holds for ${n\ge 121}$, see \cite[Theorem~1.3]{AFJ07}. This
proves~\eqref{euppersigone} for ${\tilf\ge 121}$. For ${\tilf\le 120}$
%% the inequality 
%% ${\sigma_1(\tilf)/\tilf \le 1.842\log\log(10^{7})}$
%% is true by inspection. 
one can check directly that  ${\sigma_1(\tilf)/\tilf\le 3}$ so that inequality~\eqref{euppersigone} is also
true in this case.
\end{proof}

\begin{proof}[Proof of Corollary~\ref{cseps}]
If ${|\Delta|\ge 10^{14}}$ then Lemma~\ref{lsigzersimple} implies that 
\begin{align*}
&8\left|\frac\Delta3\right|^{1/4}\sigma_0(\tilf) \le \frac8{3^{1/4}}|\Delta|^{0.442}\le \frac8{3^{1/4}\cdot10^{0.812}}|\Delta|^{1/2}\le 0.938|\Delta|^{1/2},\\
&\frac{16}3\frac{\sigma_1(\tilf)}{\tilf} \le \frac{16}3\cdot 1.842 \log\log(|\Delta|^{1/2})\le9.83 \log\log(|\Delta|^{1/2}).
\end{align*}
Substituting all this to~\eqref{enewbound}, we obtain~\eqref{enewboundsimple}. 
\end{proof}

%\section{Height bounds}

\section{An upper bound for the height of a singular unit}
\label{supper}

In this section we obtain a fully explicit version of estimate~\eqref{eupperheight}.
%Everywhere in this section~$\alpha$ is a singular modulus of discriminant~$\Delta$. 
We use the notation $\CC(\Delta)$,  $\CC_\eps(\Delta)$, $\FF$, $\zeta_3$, $\zeta_6$ introduced in Section~\ref{sceps}. 

Let~$\alpha$ be a complex algebraic number of degree~$m$ whose minimal polynomial over~$\Z$ is 
$$
P(x)=a_mx^m+\cdots+a_0=a_m(x-\alpha_1)\cdots (x-\alpha_m) \in \Z[x].
$$
Here ${\gcd(a_0,a_1, \ldots, a_m)=1}$ and ${\alpha_1, \ldots, \alpha_m\in \C}$ are the conjugates of~$\alpha$ over~$\Q$. Then the height of~$\alpha$ is defined by
$$
\height(\alpha) =\frac1m\left(\log|a_m|+\sum_{k=1}^m \log^+|\alpha_k|\right), 
$$ 
where ${\log^+(\cdot)=\log\max\{1,\cdot\}}$. 
If~$\alpha$ is an algebraic integer then 
$$
\height(\alpha) =\frac1m\sum_{k=1}^m \log^+|\alpha_k|. 
$$ 
It is known that ${\height(\alpha)=\height(\alpha^{-1})}$ when ${\alpha\ne 0}$. 

\begin{theorem}
\label{theceps}
Let~$\alpha$ be a singular unit of discriminant~$\Delta$, and~$\eps$ a real number satisfying ${0<\eps \le 4\cdot10^{-3}}$. 
Then 
\begin{equation}
\label{eupperheightbis}
\height(\alpha)\le 3\frac{\CC_\eps(\Delta)}{\CC(\Delta)}\log|\Delta|+3\log(\eps^{-1})-10.66. 
\end{equation}
\end{theorem}

Combining this with Corollary~\ref{cseps} and optimizing~$\eps$, we obtain the following consequence. 

\begin{corollary}
\label{cupper}
In the set-up of Theorem~\ref{theceps} assume that ${|\Delta|\ge 10^{14}}$. Then
\begin{align}
\label{eupdel}
\height(\alpha)&\le \frac{12A}{\CC(\Delta)}+3\log\frac{A|\Delta|^{1/2}}{\CC(\Delta)}-3.77,
%\\\label{eupdeltil}\height(\alpha)&\le \frac{\tilB}{\CC(\Delta)}+3\log\frac{A|\Delta|^{1/2}}{\CC(\Delta)}-7.69,
\end{align}
where ${A= F\log|\Delta|}$ and~$F$ is defined in~\eqref{ecapitalf}.  
\end{corollary}

%To make use of this corollary we need to estimate the quantity~$A$ in terms of~$\Delta$. 

%In this section we prove explicit versions of estimate~\eqref{eupperheight},~\eqref{etriv} and~\eqref{ecolm} for the height of a singular modulus. 

\subsection{Proof of Theorem~\ref{theceps}}

We start from some simple lemmas. %Everywhere in this section~$\alpha$ is a singular modulus of discriminant~$\Delta$. 

\begin{lemma}
\label{lanal}
For ${z\in \FF}$ we have 
$$
|j(z)|\ge 42700\bigl(\min\{|z-\zeta_3|,|z-\zeta_6|,4\cdot10^{-3}\}\bigr)^3. 
$$
\end{lemma}

\begin{proof}
This is an easy modification  of Proposition~2.2 from~\cite{BLP16}; just replace therein $10^{-3}$ by ${4\cdot10^{-3}}$.   
\end{proof}

In the next lemma we use the notation $T_\Delta$  and ${\tau(a,b,c)}$ introduced before Lemma~\ref{lgauss}. 

\begin{lemma}
\label{lliouv}
Assume that ${\Delta\ne -3}$. Let ${\tau=\tau(a,b,c)}$, where  ${(a,b,c)\in T_\Delta}$. Let~$\zeta$ be one of the numbers~$\zeta_3$ or~$\zeta_6$. 
Then
$$
|\tau-\zeta|\ge \frac{\sqrt3}{4|\Delta|}. 
$$
\end{lemma}

\begin{proof}
We have
$$
|\tau-\zeta|\ge|\Im\tau-\Im\zeta|= \left|\frac{\sqrt{|\Delta|}}{2a}-\frac{\sqrt3}{2}\right|= \frac{\bigl||\Delta|-3a^2\bigr|}{2a(\sqrt{|\Delta|}+a\sqrt3)}.
$$
Since ${\Delta\ne -3}$ we have ${\Delta\ne -3a^2}$, see item~\ref{iadelta} of Lemma~\ref{lgauss}.  Hence
$$
|\tau-\zeta|\ge \frac{1}{2a(\sqrt{|\Delta|}+a\sqrt3)}\ge \frac{\sqrt3}{4|\Delta|}, 
$$
the last inequality being again by item~\ref{iadelta} of Lemma~\ref{lgauss}. 
\end{proof}

Now we are ready to prove Theorem~\ref{theceps}. 

\begin{proof}[Proof of Theorem~\ref{theceps}]
Let ${\alpha=\alpha_1, \alpha_2, \ldots, \alpha_m\in \C}$ be the conjugates of~$\alpha$ over~$\Q$. Then ${m=\CC(\Delta)}$ and ${\alpha_1, \ldots, \alpha_m}$ is the full list of singular moduli of discriminant~$\Delta$. Write them as ${j(\tau_1), \ldots, j(\tau_m)}$, where ${\tau_1, \ldots, \tau_m\in \FF}$. 

Since~$\alpha$ is a unit, we have 
$$
\height(\alpha)=\height(\alpha^{-1})  = \frac1m \sum_{k=1}^m\log^+|\alpha_k^{-1}|,
$$
%where ${\log^+(x)=\max\{\log x, 0\}}$. 
Hence 
{\footnotesize
\begin{align}
\height(\alpha) &= \frac{1}{\CC(\Delta)}\sum_{k=1}^m\log^+|j(\tau_k)^{-1}|\nonumber\\
\label{etwosums}
&= \frac{1}{\CC(\Delta)}\left(\sum_{\genfrac{}{}{0pt}{}{1\le k\le m}{\min\{|\tau_k-\zeta_3|, |\tau_k-\zeta_6|\}< \eps}}+\sum_{\genfrac{}{}{0pt}{}{1\le k\le m}{\min\{|\tau_k-\zeta_3|, |\tau_k-\zeta_6|\}\ge \eps}}\right)\log^+|j(\tau_k)^{-1}|.
\end{align}
}%
We estimate each of the two sums separately.

Since ${\eps\le 4\cdot10^{-3}}$, Lemma~\ref{lanal} implies that each term in the second sum satisfies 
$$
\log^+|j(\tau_k)^{-1}| \le 3\log(\eps^{-1})-\log42700 \le 3\log(\eps^{-1})-10.66. 
$$
Hence
$$
\sum_{\genfrac{}{}{0pt}{}{1\le k\le m}{\min\{|\tau_k-\zeta_3|, |\tau_k-\zeta_6|\}\ge \eps}}\log^+|j(\tau_k)^{-1}|\le (\CC(\Delta)-\CC_\eps(\Delta))\bigl(3\log(\eps^{-1})-10.66\bigr). 
$$
Since ${\eps\le 4\cdot10^{-3}}$ we have
 ${3\log(\eps^{-1})>10.66}$, which implies that 
\begin{equation}
\label{esecondsum}
\sum_{\genfrac{}{}{0pt}{}{1\le k\le m}{\min\{|\tau_k-\zeta_3|, |\tau_k-\zeta_6|\}\ge \eps}}\log^+|j(\tau_k)^{-1}|\le \CC(\Delta)\bigl(3\log(\eps^{-1})-10.66\bigr). 
\end{equation}
As for the first sum, Lemmas~\ref{lanal} and~\ref{lliouv} imply that each term in this sum satisfies
$$
\log^+|j(\tau_k)^{-1}| \le \max\left\{0, 3\log\frac{4|\Delta|}{\sqrt3}-\log42700\right\} \le 3\log|\Delta|.  
$$ 
Note that we may use here Lemma~3.4 because the only singular modulus of discriminant $-3$ is~$0$, which is not a unit. 

Since the first sum has $\CC_\eps(\Delta)$ terms, this implies the estimate 
\begin{equation}
\label{efirstsum}
\sum_{\genfrac{}{}{0pt}{}{1\le k\le m}{\min\{|\tau_k-\zeta_3|, |\tau_k-\zeta_6|\}< \eps}}\log^+|j(\tau_k)^{-1}|\le 3\CC_\eps(\Delta)\log|\Delta|. 
\end{equation}
Substituting~\eqref{esecondsum} and~\eqref{efirstsum} into~\eqref{etwosums}, we obtain~\eqref{eupperheightbis}.
\end{proof}

\subsection{Proof of Corollary~\ref{cupper}}

To prove the corollary we need a lower bound for the quantity~$F$ defined in Theorem~\ref{thceps} and an upper bound for the class number $\CC(\Delta)$.  

\begin{lemma}
\label{lf}
Assume that ${|\Delta|\ge 10^{14}}$. Then   ${F\ge |\Delta|^{0.34/\log\log(|\Delta|^{1/2})}}$ and ${F\ge 18.54\log\log(|\Delta|^{1/2})}$.  
\end{lemma}

\begin{proof}
Define, as usual 
\begin{equation}
\label{ethetapi}
\vartheta(x)=\sum_{p\le x}\log p, \qquad \pi(x)=\sum_{p\le x}1. 
\end{equation}
Then 
\begin{align}
\vartheta(x) & \le 1.017 x &&(x>0), \nonumber\\
\label{elowerpi}
\pi(x) &\ge \frac{x}{\log x} && (x\ge 17),
\end{align} 
see \cite{RS62}, Theorem~9 on page~71 and Corollary~1 after Theorem~2 on page~69. Estimate~\eqref{elowerpi} implies that
\begin{equation}
\label{elowerpibis}
\pi(x) \ge 0.99995\frac{x}{\log x}  \qquad (x\ge 13).
\end{equation}
Setting here 
$$
x=\frac{\log(|\Delta|^{1/2})}{1.017} , \qquad N =\prod_{p\le x}p,
$$
we obtain ${N\le |\Delta|^{1/2}}$ and 
$$
\omega(N) =\pi(x) \ge \frac{0.99995\log(|\Delta|^{1/2})}{1.017\log\log(|\Delta|^{1/2})}. 
$$
Note that ${x\ge \bigl(\log(10^{7})\bigr)/1.017>15}$, so we are allowed to use~\eqref{elowerpibis}. 
We obtain 
$$
F\ge 2^{\omega(N)} \ge |\Delta|^{\frac{0.99995\log2}{2\cdot1.017\log\log(|\Delta|^{1/2})}}\ge |\Delta|^{0.34/\log\log(|\Delta|^{1/2})},
$$
proving the first estimate.

To prove the second estimate, we deduce from the first estimate that 
\begin{equation}
\label{elog-log}
\log F-\log\log\log(|\Delta|^{1/2})\ge 0.68 \frac{u}{\log u}-\log\log u,
\end{equation} 
where we set ${u=\log(|\Delta|^{1/2})}$. The right-hand side of~\eqref{elog-log}, viewed as a function in~$u$, is increasing for ${u\ge \log(10^{7})}$. Hence 
$$
\log F-\log\log\log(|\Delta|^{1/2})\ge 0.68 \frac{\log(10^{7})}{\log \log(10^{7})}-\log\log \log(10^{7}) \ge 2.92,
$$
and ${F\ge e^{2.92}\log\log(|\Delta|^{1/2}) \ge 18.54 \log\log(|\Delta|^{1/2})}$. 
\end{proof}

\begin{lemma}
\label{lcd}
For ${\Delta\ne -3,-4}$ we have 
$$
\CC(\Delta) \le \pi^{-1}|\Delta|^{1/2}(2+\log|\Delta|).
$$ 
%If ${|\Delta|\ge 10^{14}}$ then  ${\CC(\Delta)\le 0.34 |\Delta|^{1/2}\log|\Delta|}$. 
\end{lemma}

\begin{proof}
This follows from Theorems~10.1 and~14.3 in \cite[Chapter~12]{Hu82}. Note that in~\cite{Hu82} the right-hand side has an extra factor $\omega/2$, where~$\omega$ is the number of roots of unity in the imaginary quadratic order of discriminant~$\Delta$. Since we assume that ${\Delta \ne -3,-4 }$, we have ${\omega=2}$, so we may omit this factor. 
\end{proof}

%The following is an immediate consequence of Theorem~\ref{thceps} and Proposition~\ref{pheceps}. 

\begin{proof}[Proof of Corollary~\ref{cupper}]
Substituting the estimate for $\CC_\eps(\Delta)$ from~\eqref{enewboundsimple} into~\eqref{eupperheightbis}, we obtain the estimate
\begin{align*}
\height(\alpha)&\le 3A\frac{9.83|\Delta|^{1/2} \eps^2\log\log(|\Delta|^{1/2})+3.605|\Delta|^{1/2}\eps+ 4}{\CC(\Delta)}
\\&\hphantom{\le}
+3\log(\eps^{-1})-10.66
\end{align*}
with ${A=F\log|\Delta|}$. Specifying 
$$
\eps=0.27\frac{\CC(\Delta)}{A|\Delta|^{1/2}}
$$
(this is a nearly optimal value, and it satisfies ${\eps\le 4\cdot10^{-3}}$ as verified below), we obtain, using Lemmas~\ref{lf} and~\ref{lcd},
\begin{align*}
\height(\alpha)&\le  3\cdot 9.83 \cdot (0.27)^2\frac{\log\log(|\Delta|^{1/2})}F\frac{\CC(\Delta)}{|\Delta|^{1/2}\log|\Delta|}+3\cdot3.605\cdot0.27\\ 
&\hphantom{\le}+\frac{12A}{\CC(\Delta)} +3\log\frac{A|\Delta|^{1/2}}{\CC(\Delta)}-3\log0.27-10.66\\
&\le \frac{12A}{\CC(\Delta)} +3\log\frac{A|\Delta|^{1/2}}{\CC(\Delta)}+ \frac{3\cdot 9.83 \cdot (0.27)^2\cdot0.34}{18.54}+3\cdot3.605\cdot0.27\\ 
&\hphantom{\le}-3\log0.27-10.66\\
&\le \frac{12A}{\CC(\Delta)} +3\log\frac{A|\Delta|^{1/2}}{\CC(\Delta)}-3.77,
\end{align*}
as wanted.

We only have to verify that ${\eps\le 4\cdot10^{-3}}$. We have
${F\ge 256}$ when ${|\Delta|\ge 10^{14}}$. Using Lemma~\ref{lcd}, we obtain 
\begin{align*}
\eps=\frac{0.27\CC(\Delta)}{|\Delta|^{1/2}\log |\Delta|}\frac1F
\le 0.27\pi^{-1}\frac{2+\log(10^{14})}{\log(10^{14})}\cdot \frac1{256}
<4\cdot 10^{-4}.
\end{align*}
The proof is complete.
\end{proof}

\section{Lower bounds for the height of a singular modulus}
\label{slower}

Now we establish explicit lower bounds of the form~\eqref{ecolm} and~\eqref{etriv}. 

\subsection{The ``easy'' bound}

We start by proving a bound of the form~\eqref{etriv}. 

\begin{proposition}
\label{ptriv}
Let~$\alpha$ be a singular modulus of discriminant~$\Delta$. 
Assume that ${|\Delta|\ge 16}$. Then 
\begin{equation}
\label{etrivbis}
\height(\alpha) \ge  \frac{\pi|\Delta|^{1/2}-0.01}{\CC(\Delta)}.
\end{equation}
\end{proposition}

We need a simple lemma. 

\begin{lemma}
\label{lttsn}
For ${z\in \FF}$ with imaginary part~$y$ we have 
$$
\bigl||j(z)|-e^{2\pi y}\bigr|\le 2079.
$$ 
If ${y\ge 2}$ then we also have ${|j(z)|\ge 0.992 e^{2\pi y}}$. 
\end{lemma}

\begin{proof}
The first statement is Lemma~1 of~\cite{BMZ13}, and the second one is an immediate consequence. 
\end{proof}

\begin{proof}[Proof of Proposition~\ref{ptriv}]
One of the conjugates of~$\alpha$ over~$\Q$ is equal to ${j((b+\sqrt\Delta)/2)}$, with ${b=1}$ for~$\Delta$ odd, and ${b=0}$ for~$\Delta$ even;  it corresponds to the element ${(1,b,(-\Delta+b^2)/4)}$ of the set $T_\Delta$. 
Hence
$$
\height(\alpha) \ge  \frac{\log|j((b+\sqrt\Delta)/2)|}{\CC(\Delta)}.
$$
Using Lemma~\ref{lttsn}, we obtain
$$
\log|j((b+\sqrt\Delta)/2)|\ge \pi|\Delta|^{1/2}+\log 0.992 \ge \pi|\Delta|^{1/2}-0.01.
$$
Whence the result. 
\end{proof}

\subsection{The ``hard'' bound}

\label{sshard}

We are left with  bound~\eqref{ecolm}. We are going to prove the following. 

\begin{proposition}
  \label{phlb2}
Let~$\alpha$ be a singular modulus of discriminant~$\Delta$. %${\Delta=Df^2}$. 
Then
\begin{equation}
\label{ecolmbis}
\height(\alpha) \ge \frac{3}{\sqrt{5}}\log|\Delta| -9.79. 
\end{equation}
\end{proposition}

The proof of Proposition~\ref{phlb2} relies on the
fact that it is possible to evaluate the Faltings height of an elliptic curve
with complex multiplication precisely, due to the work of Colmez~\cite{Colmez} and Nakkajima-Taguchi~\cite{NT}; for an exact statement see  \cite[Lemma 4.1]{Ha10}. 

Let~$E$ be an elliptic curve with CM by an order of discriminant~$\Delta$.  
We let $\height_F(E)$ denote the stable Faltings height of~$E$ (using Deligne's normalization~\cite{De85}). 
The above-mentioned explicit formula for $\height_F(E)$ is used in~\cite{HJM:6} to obtain the lower bound  
\begin{equation}
\label{efaltlowerold}
\height_F(E) \ge \frac{1}{4\sqrt5} \log|\Delta| - 5.93,
\end{equation}
see Lemma 14(ii) therein. Unfortunately, this bound is numerically too weak for our purposes.

Proposition~\ref{phlb2} will be deduced from the following numerical refinement of~\eqref{efaltlowerold}.

\begin{proposition}
\label{pfaltlower}
Let~$E$ be an elliptic curve with CM by an order of discriminant~$\Delta$. %${\Delta=Df^2}$ 
Then 
\begin{equation}
\label{efaltlowernew}
\height_F(E) \ge \frac{1}{4\sqrt5} \log|\Delta| -\gamma -\frac{\log(2\pi)}{2}-\left(\frac1{2\sqrt5}-\frac16\right)\log2,
\end{equation}
where ${\gamma=0.57721\ldots}$ is the Euler constant. 
\end{proposition}

%Proposition~\ref{phlb2} is indeed in easy consequence of this statement. 

Let us first show how Proposition~\ref{pfaltlower} implies Proposition~\ref{phlb2}.

\begin{proof}[Proof of Proposition~\ref{phlb2} (assuming Proposition~\ref{pfaltlower})]
Let~$E$ be an elliptic curve with ${j(E)=\alpha}$. 
We only need to relate $\height_F(E)$ to $\height(j(E))$. For this purpose we use Lemma~7.9 of Gaudron and R\'emond~\cite{GaudronRemond:periodes}\footnote{The reader should be warned that our $\height_F(E)$ is denoted $\height(E)$ in~\cite{GaudronRemond:periodes}.}.
In our
notation  they
show that 
\begin{equation}
\label{egr}
\height_F(E)\le \height(j(E))/12 - 0.72 
\end{equation}
A quick calculation yields  our
claim. 
\end{proof}

To prove Proposition~\ref{pfaltlower} we need a technical lemma. %We again use the standard notation~\eqref{ethetapi}. 
Set 
\begin{equation}
\label{elambda}
\lambda=\frac12-\frac1{2\sqrt5},
\end{equation}
and define the additive arithmetical functions $\beta(n)$ and $\delta(n)$ by  
\begin{equation}
\label{ebetadelta}
\beta(p^k) =\frac{\log p}{p+1}\frac{1-p^{-k}}{1-p^{-1}}, \quad \beta(n)=\sum_{p^k\|n}\beta(p^k), \quad \delta(n)= \lambda\log n -\beta(n).
\end{equation}

\begin{lemma}
\label{ltechn} 
For every positive integer~$n$ we have 
$$
\delta(n) \ge\delta(2)= \left(\frac16-\frac1{2\sqrt5}\right)\log2. 
$$
\end{lemma}

\begin{proof}
Since ${1/3>\lambda>1/4}$, we have ${\delta(2)<0}$ and  ${\delta(p)>0}$ for all primes ${p\ge 3}$. Also, for ${k\ge 1}$ and any prime~$p$ we have 
$$
\delta(p^{k+1})-\delta(p^k)= \left(\lambda-\frac1{p^k(p+1)}\right)\log p >0.
$$
Since ${\delta(4)>0}$, this proves that ${\delta(p^k)>0}$ for every prime power ${p^k\ne 2}$, whence the result. 
\end{proof}

Proposition~\ref{pfaltlower} is an immediate consequence of Lemma~\ref{ltechn} and the following statement.  

\begin{proposition}
\label{pfaltlowerf}
In the set-up of Proposition~\ref{pfaltlower} we have
\begin{equation}
\label{efaltlowernewf}
\height_F(E) \ge \frac{1}{4\sqrt5} \log|\Delta| + \lambda \log f- \beta(f) -\gamma -\frac{\log(2\pi)}{2}.
\end{equation}
\end{proposition}
Since 
$$
\lambda \log f- \beta(f)\ge-\left(\frac1{2\sqrt5}-\frac16\right)\log2
$$
by Lemma~\ref{ltechn}, this implies Proposition~\ref{pfaltlower}. 

\begin{proof}[Proof of Proposition~\ref{pfaltlowerf}]
Write ${\Delta=Df^2}$ with~$D$ the fundamental discriminant and~$f$ the conductor. Define 
$$
e_f(p) = \frac{1-\chi(p)}{p-\chi(p)}\frac{1-p^{-\ord_p(f)}}{1-p^{-1}}, \qquad c(f) = \frac 12 \left(\sum_{p\mid f} e_f(p)\log p\right),
$$
where ${\chi(\cdot)=(D/\cdot)}$ is Kronecker's symbol.

In the proof of Lemma~14 of~\cite{HJM:6}\footnote{Note that our~$D$ is written~$\Delta$ in~\cite{HJM:6}.},
the stable Faltings height of~$E$  is estimated as 
\begin{align*}
%\label{ehjm}
\height_F(E) &\ge \frac{1}{4\sqrt5}\log|D| + \frac 12 \log f
-c(f)
-\gamma -\frac{\log(2\pi)}{2},\\
&=\frac{1}{4\sqrt5}\log|\Delta| + \lambda \log f
-c(f)
-\gamma -\frac{\log(2\pi)}{2}.
\end{align*}
%where~$\lambda$ is defined in~\eqref{elambda}. 
Thus, to establish~\eqref{efaltlowernewf}, we only have to prove that ${c(f)\le \beta(f)}$. 
We have 
$$
\frac{1-\chi(p)}{p-\chi(p)} = 
\begin{cases}
0, &\chi(p)=1,\\
1/p, &\chi(p)=0,\\
 2/(p+1), &\chi(p)=-1. 
\end{cases}
$$
Hence
$$
\frac{1-\chi(p)}{p-\chi(p)} \le \frac2{p+1}
$$
in any case. This implies that ${c(f)\le \beta(f)}$. %, where ${\beta(\cdot)}$ is the arithmetical function defined in~\eqref{ebetadelta}.  Therefore inequality~\eqref{einequalitytoprove} follows from Lemma~\ref{ltechn}. 
The proposition is proved. 
\end{proof}

\section{The estimate \texorpdfstring{${|\Delta|< 10^{15}}$}{|Delta|<10to15}}
\label{stenfourteen}

In this section we obtain the first explicit upper bound for the
discriminant of a singular unit.

\begin{theorem}
  \label{thhighrange}
Let~$\Delta$  be the discriminant of a singular unit. Then
  ${|\Delta|<10^{15}}$. 
\end{theorem}

{\sloppy

Throughout this section~$\Delta$ is the discriminant of a singular unit~$\alpha$, and  we assume that  ${X=|\Delta|\ge 10^{15}}$, as otherwise there is nothing to prove.  %This means, in particular,  that we may use  estimate~\eqref{eupdel}. 
Our principal tools will be the upper estimate~\eqref{eupdel} %(which we are allowed to use because ${X\ge 10^{14}}$), 
and  the lower estimates~\eqref{etrivbis},~\eqref{ecolmbis}. We reproduce them here for convenience:
\begin{align}
\label{eupx}
\height(\alpha)&\le \frac{12A}{\CC(\Delta)}+3\log\frac{AX^{1/2}}{\CC(\Delta)}-3.77,\\
\label{etrivx}
\height(\alpha)&\ge  \frac{\pi X^{1/2}-0.01}{\CC(\Delta)},\\
\label{ecolmx}
\height(\alpha)&\ge \frac{3}{\sqrt5}\log X-9.79.
\end{align}
Note that our assumption ${X\ge 10^{15}}$ implies that the right-hand side of~\eqref{ecolmx} is positive.

}

\subsection{The main inequality}
\label{ssinequality}

Recall that ${A= F\log X}$. 
Minding~$0.01$ in~\eqref{etrivx} we deduce from~\eqref{eupx},~\eqref{etrivx} and~\eqref{ecolmx} the inequality
$$
\frac{12A}{\CC(\Delta)}+3\log\frac{AX^{1/2}}{\CC(\Delta)}-3.76 \ge \max\left\{\frac{\pi X^{1/2}}{\CC(\Delta)}, \frac{3}{\sqrt5}\log X-9.78\right\}. 
$$
Denoting
\begin{equation}
\label{ey}
Y=\max\left\{\frac{\pi X^{1/2}}{\CC(\Delta)}, \frac{3}{\sqrt5}\log X-9.78\right\}, 
\end{equation}
we re-write this as 
\begin{equation}
\label{enotyet}
\frac{12A/\CC(\Delta)}{Y}+\frac{3\log A-3.76}Y+ \frac{\log(X^{1/2}/\CC(\Delta))}{Y} \ge 1. 
\end{equation}
Note that ${3\log A-3.76>0}$, because ${A\ge  \log X\ge \log(10^{15}) >30}$. Hence we may replace~$Y$ by ${\frac{3}{\sqrt5}\log X-9.78}$ in the middle term of the left-hand side in~\eqref{enotyet}. Similarly, in the first term we may replace~$Y$ by ${\pi X^{1/2}/\CC(\Delta)}$, and in the third term we may replace ${X^{1/2}/\CC(\Delta)}$ by ${\pi^{-1}Y}$. We obtain 
\begin{equation}
\label{einequality}
12\pi^{-1} AX^{-1/2}+ \frac{3\log A-3.76}{\frac{3}{\sqrt5}\log X-9.78} + 3\frac{\log(\pi^{-1}Y)}{Y}\ge 1. 
\end{equation}

%}

To show that~\eqref{einequality} is not possible for ${X\ge 10^{15}}$, we will bound from above each of the three terms in its left-hand side.  To begin with, we bound~$A$.

\subsection{Bounding~$F$ and~$A$}

Recall that ${F=\max\{2^{\omega(a)}: a\le X^{1/2}\}}$ and ${A=F\log X}$. 

Let  
${N_1 = 2\cdot3\cdot5\cdots 1129}$  be the product of the first $189$ prime numbers. Define the real number~$c_1$ from 
\begin{align*}
\omega(N_1)&=\frac{\log N_1}{\log\log N_1 -c_1}.
\end{align*}
A calculation shows that ${c_1 < 1.1713142}$. 
Robin  \cite[Théorème~13]{Ro83} proved that
\begin{equation*}
%\label{erobome}
\omega(n) \le \frac{\log n}{\log\log n-c_1} 
\end{equation*}
for ${n\ge 26}$. 
This implies that 
\begin{align}
\label{eboundf}
\frac{\log F}{\log 2} &\le \frac12\frac{\log X}{\log\log X-c_1-\log2} ,\\
\label{ebounda}
\log A &\le \frac{\log2}2\frac{\log X}{\log\log X-c_1-\log2}+\log\log X. 
\end{align}
Indeed, the function 
$$
g(x)= \frac{\log x}{\log\log x-c_1}
$$
is strictly increasing for ${x\ge 6500}$ and ${g(6500)>8}$. 
%If ${X\ge 10^{30}}$ then ${X^{1/2}>6500}$. 
If ${a\le X^{1/2}}$ then either ${a\le 6500}$ in which case 
${\omega(a) \le 5<g(6500)<g(X^{1/2})}$
(recall that ${X\ge 10^{15}}$), or ${6500<a\le X^{1/2}}$, in which case 
${\omega(a) \le g(a)\le g(X^{1/2})}$.
Thus, in any case we have 
$$
\omega(a) \le g(X^{1/2})= \frac12\frac{\log X}{\log\log X-c_1-\log2}, 
$$
which proves~\eqref{eboundf}. The estimate~\eqref{ebounda} is an immediate consequence of~\eqref{eboundf}.

\subsection{Bounding the first term in~\eqref{einequality}}
Using~\eqref{ebounda}, we estimate
$$
\frac{\log (AX^{-1/2})}{\log X}\le u_0(X), 
$$
where
\begin{align*}
u_0(x)&= \frac{\log2}2\frac{1 }{\log\log x-c_1-\log2}+\frac{\log\log x}{\log x}-\frac12. 
\end{align*}
The function $u_0(x)$ is decreasing for ${x\ge 10^{10}}$. Hence 
for ${X\ge 10^{15}}$ we have 
$$
\frac{\log (AX^{-1/2})}{\log X}\le u_0(10^{15}) <-0.1908. 
$$
This proves the estimate 
\begin{equation*}
%\label{efirstterm}
AX^{-1/2}<  10^{15\cdot(-0.1908)}<0.0014
\end{equation*}
for ${X\ge 10^{15}}$.

\subsection{Bounding the second term in~\eqref{einequality}}
Using~\eqref{ebounda}, we estimate 
$$
\frac{3\log A-3.76}{\frac{3}{\sqrt5}\log X-9.78}\le u_1(X)u_2(X), 
$$
where
\begin{align*}
u_1(x) &= \frac{3\log2}2\frac{1}{\log\log x -c_1-\log2}+\frac{3\log\log x-3.76}{\log x},\\
u_2(x)& = \left(\frac{3}{\sqrt5}-\frac{9.78}{\log x}\right)^{-1}. 
\end{align*}
Both functions $u_1(x)$ and $u_2(x)$ are decreasing for ${x\ge 10^{10}}$. Hence, for ${X\ge 10^{15}}$ we have
\begin{equation*}
%\label{esecondterm}
\frac{3\log A-3.76}{\frac{3}{\sqrt5}\log X-9.78}\le u_1(10^{15})u_2(10^{15}) <0.7734. 
\end{equation*}

\subsection{Bounding the third term in~\eqref{einequality}}

The function ${x\mapsto (\log x)/x}$ is decreasing for ${x\ge e}$. Since for ${X\ge 10^{15}}$ we have 
$$
\pi^{-1}Y\ge \pi^{-1}\left(\frac{3}{\sqrt5}\log X-9.78\right)\ge e,
$$
we have, for ${X\ge 10^{15}}$, the estimate 
$$
\frac{\log(\pi^{-1}Y)}{Y} \le u_3(X), 
$$
where 
$$
u_3(x) = \frac{\log\left(\pi^{-1}\bigl(\frac{3}{\sqrt5}\log x-9.78\bigr)\right)}{\frac{3}{\sqrt5}\log x-9.78}. 
$$
Moreover, the function $u_3(x)$ is decreasing for ${x\ge 10^{15}}$, which implies that 
\begin{equation*}
%\label{ethirdterm}
\frac{\log(\pi^{-1}Y)}{Y}\le u_3(X) \le u_3(10^{15}) < 0.0672
 \end{equation*}
for ${X\ge 10^{15}}$. 

\subsection{Summing up}

Now, when ${X\ge 10^{15}}$, we can combine the above estimates and bound the left-hand side of~\eqref{einequality} by
$$
12\pi^{-1}\cdot0.0014 + 0.7734 + 3 \cdot 0.0672 <0.981.
$$
Hence, for ${X\ge 10^{15}}$ we cannot have~\eqref{einequality}. This contradiction completes the proof of Theorem~\ref{thhighrange}. 

%This proves that ${X<10^{15}}$. 

\section{Handling the mid-range \texorpdfstring{$10^{10}\le |\Delta|< 10^{15}$}{10to10<Delta<10to15}}
\label{sec:midrange}

%In this section we reduce the bound $10^{15}$ from the previous section to $10^{10}$. 

In this section we rule out the existence of singular units with discriminants in the
mid-range ${[10^{10}, 10^{15})}$, improving thereby the bound from the previous section.

\begin{theorem}
  \label{thmidrange}
Let~$\Delta$  be the discriminant of a singular unit. Then
  ${|\Delta|\notin [10^{10},10^{15})}$. 
\end{theorem}

In Section~\ref{sceps} we estimate trivially ${2^{\omega(a)}\le F}$. One might expect to do better by estimating the average order rather than the maximal order of the arithmetical function ${2^{\omega(n)}}$. This is accomplished in Subsection~\ref{ssdelange} and  allows us to obtain, in Subsection~\ref{ssnewceps}, a new bound for ${\CC_\eps(\Delta)}$ in the range ${10^{10}\le |\Delta|< 10^{15}}$. Using this, Theorem~\ref{thmidrange} is proved in Subsection~\ref{ssproofmidrange} by  an argument similar to the proof of Theorem~\ref{thhighrange}. 

Throughout this section~$n$ denotes a positive integer. 

\subsection{Average order of the function $2^{\omega(n)}$ on subintervals of ${[0,2\cdot10^7]}$}
\label{ssdelange}
For a positive  real number~$x$ set 
$$
S(x)= \sum_{n\le x} 2^{\omega(n)}.
$$
We define $S(0)=0$.
As Theorem II.6.1 from~\cite{Te15} suggests, the function $S(x)$ can be well approximated by the function
$$
g(x)= \lambda_0x\log x+\lambda_1 x,
$$
where 
\begin{align*}
\lambda_0&= \zeta(2)^{-1}=0.607927101854026\dots,\\%62866327677925837\dots,\\ 
\lambda_1 &= -2\frac{\zeta^{\prime}(2)}{\zeta(2)^{2}}+\frac{2\gamma-1}{\zeta(2)}=0.786872460166245\dots,%53945352846402194\dots,
\end{align*}
and~$\gamma$ is the Euler constant. The function~$g$ is increasing on
$[1,\infty)$. 

%To estimate the error term ${|S(x)-g(x)|}$ for all~$x$ one needs to use hard tools of the Analytic Number Theory, or assume the Riemann Hypothesis. However, on a limited range of~$x$ such an estimate can be obtained by a quick calculation. 

As already mentioned in the introduction, the error term ${|S(x)-g(x)|}$ can be estimated by the Selberg-Delange method~\cite[Chapter~II.5]{Te15}, but on our
limited range it is more advantageous to obtain an optimal error term by a computer-assisted
calculation.

\begin{proposition}
\label{proposition::selbergdelange}
For ${2\le x \le 2\cdot 10^7}$ we have 
\begin{equation}
\label{esqrt}
g(x)- 1.010x^{1/2} \le S(x) \le g(x)+ 0.712x^{1/2}, 
\end{equation}
and for ${4\cdot10^4\le x\le 2\cdot 10^7}$ we have 
\begin{equation}
\label{esqrtlog}
g(x)- 2.267\frac{x^{1/2}}{\log x} \le S(x) \le g(x)+ 2.598\frac{x^{1/2}}{\log x}. 
\end{equation}
\end{proposition}

\begin{proof}
Set
\begin{align*}
c_1&=\max_{2\le n\le 2\cdot 10^7}\frac{S(n)-g(n)}{\sqrt n}, \\
c_2&=\max_{2\le n\le 2\cdot 10^7}\frac{g(n+1)-S(n)}{\sqrt n},\\
c_3&=\max_{4\cdot10^4\le n\le 2\cdot 10^7}\frac{(S(n)-g(n))\log n}{\sqrt n}, \\
c_4&=\max_{4\cdot10^4\le n\le 2\cdot 10^7}\frac{(g(n+1)-S(n))\log n}{\sqrt n}.
\end{align*}
%where here and below $n$ denotes an integer. 
Then for  ${2\le n\le 2\cdot 10^7}$ we have
$$
S(n) \le g(n)+c_1\sqrt n,   \qquad S(n) \ge g(n+1)-c_2\sqrt {n}.
$$
Hence for ${2\le x\le 2\cdot 10^7}$ we have 
\begin{align*}
S(x)&= S(\lfloor x \rfloor) \le g(\lfloor x \rfloor)+c_1\sqrt{\lfloor x\rfloor} \le g(x)+c_1\sqrt x,\\
S(x)&= S(\lfloor x \rfloor) \ge g(\lfloor x \rfloor+1)-c_2\sqrt{\lfloor x\rfloor} \ge g(x)-c_2\sqrt x.
\end{align*}
In a similar way 
we show that 
for ${4\cdot10^4\le x\le 2\cdot 10^7}$ 
$$
g(x)-c_4\frac{\sqrt x}{\log x}\le S(x) \le g(x)+c_3\frac{\sqrt x}{\log x}
$$
having used that $x\mapsto \sqrt{x}/\log x$ is increasing on
$(e^2,\infty)$.
A computer-assisted calculation shows that 
$$
c_1\le 0.712, \quad c_2\le 1.010, \quad
c_3\le 2.598, \quad c_4\le 2.267.
%% c_3\le 3.781, \quad c_4\le 3.751.
$$
We verify this by means of a
\textsf{SAGE}~\cite{sagemath} script\footnote{A link to the  script
  \textsf{prop6\_2.sage} is on  the second named
author's homepage. The running time is roughly 30 minutes on a regular
desktop (Intel Xeon CPU E5-1620 v3, 3.50GHz, 32GB RAM).} using the interval arithmetic \textsf{MPFI} package~\cite{MPFI}.
\end{proof}

%The following consequence is immediate. 
\begin{corollary}
Let~$A$ and~$B$ be positive real numbers satisfying 
$$
0<A<B\le 2\cdot 10^7, \qquad B\ge 1.
$$
Then 
\begin{equation}
\label{esqrtab}
\sum_{A<n\le B}2^{\omega(n)}\le \lambda_0 (B-A)(1+\log B)+\lambda_1(B-A)+1.722B^{1/2}
\end{equation}
If, in addition to this, 
$
{A\ge 4\cdot10^4}, 
$
then 
\begin{equation}
\label{esqrtlogab}
\sum_{A<n\le B}2^{\omega(n)}\le \lambda_0 (B-A)(1+\log B)+\lambda_1(B-A)+4.865\frac{B^{1/2}}{\log B}. 
\end{equation}
\end{corollary}

\begin{proof}
In general we have
\begin{equation*}
S(B)-S(A)=  \sum_{A<n\le B} 2^{\omega(n)} = \sum_{n =
    \lfloor{A}\rfloor+1}^{\lfloor{B}\rfloor}2^{\omega(n)} =
  S(\lfloor B\rfloor) - S(\lfloor A\rfloor).
\end{equation*}
Note that since $B>A>0$ we remark that
\begin{equation}
\label{eremark}
B\log B - A \log A \le (B-A)(1+\log B).
\end{equation}
When ${A\ge 2}$ estimate~\eqref{esqrtab} follows immediately
from~\eqref{esqrt} and~\eqref{eremark}.

Let us assume  that ${A<2}$, hence $\lfloor A\rfloor$ is $0$ or $1$
and $S(\lfloor A\rfloor)=0$ or $1$, respectively. For $B\ge 2$ we find
$$
\sum_{A<n\le B}2^{\omega(n)} \le \lambda_0 B\log
B+\lambda_1B+0.712B^{1/2} -S(\lfloor A\rfloor),
$$
and one easily verifies that 
$$
-S(\lfloor A\rfloor)+\lambda_0A(1+\log B) +\lambda_1A - \lambda_0B%% < 2\lambda_0(1+\log
%% B)+2\lambda_1 -\lambda_0 B
\le 1.010B^{1/2}
$$
by considering the cases $A\in (0,1)$ and $A\in [1,2)$ separately.
  This implies (\ref{esqrtab}). 
And if ${B<2}$, then 
${\sum_{A<n\le B} 2^{\omega(n)}\le 1\le 1.722
B^{1/2}}$.  
As $B\ge 1$, again we obtain (\ref{esqrtab}). 
%proving~\eqref{esqrtab} in the case ${A<2}$ as well.

Finally, for ${4\cdot10^4\le A<B\le 2\cdot 10^7}$ estimate~\eqref{esqrtlogab}
follows %immediately 
from~\eqref{esqrtlog} and~\eqref{eremark}.  
\end{proof}

\subsection{Bounding $\CC_\eps(\Delta)$ for ${10^{10}\le|\Delta|< 10^{15}}$}
\label{ssnewceps}

Now we can obtain a cardinal refinement of Theorem~\ref{thceps}  for discriminants in the range ${10^{10}\le |\Delta|< 10^{15}}$.  We need a technical lemma using our
notation~\eqref{eq:standardnotation}. 
%% \begin{equation}
%% \label{evarsigma}
%% \varsigma_k(n)=\sum_{d\mid n}2^{\omega(d)}d^{-k}.
%% \end{equation}

\begin{lemma}
  \label{lem:varsigma01}
  Let $n$ be an integer with $1\le n\le 3.2\cdot 10^7$.
  \begin{enumerate}[label={(\roman*)}]
  \item 
  \label{isigone} We have
    $\sigma_1(n)/n \le \sigma_1(21621600)/21621600= 3472/715$. 
  \item 
  \label{isigzero} We have
    $\sigma_0(n)\le 8.5 n^{1/4}$. 
  \end{enumerate}
\end{lemma}
\begin{proof}
This can be proved in at least two ways. In our relatively small range, we can
perform a quick computer-assisted calculation\footnote{A link to the \textsf{PARI} script \textsf{lemma6\_4.gp} is on the second-named
author's homepage. The running time is under 2 minutes on a regular
desktop (Intel Xeon CPU E5-1620 v3, 3.50GHz, 32GB RAM).}.  Alternatively, one can use the
\textit{On-Line Encyclopedia of Integer Sequences}~\cite{oeis}. Inspecting T. D. Noe and D. Kilminster's
table in entry \textsf{A004394}, one can deduce~\ref{isigone}. Similarly,~\ref{isigzero} follows from inspection of T. D. Noe's table in \textsf{A002182}.
\end{proof}

\begin{proposition}
  \label{prop:Cepsmidrange}
In the set-up of Theorem~\ref{thceps} assume that 
\begin{equation}
\label{emidrange}
{10^{10} \le |\Delta|< 10^{15}}.
\end{equation}
Then
\begin{align*}
\CC_\eps(\Delta) &\le (8\eps^2 + 0.811\eps)|\Delta|^{1/2} \log|\Delta|
+(28\eps^2 +2.829\eps)|\Delta|^{1/2}  \\
&\hphantom{\le}+89\eps {|\Delta|^{3/8}}+
%%1278\frac{|\Delta|^{1/2}}{\log|\Delta|}\eps+
31.06\frac{|\Delta|^{1/4}}{\log|\Delta|}.% \mathcal{E}
\end{align*}
%% with
%% \begin{equation*}
%%   \mathcal{E}=  \left\{
%%   \begin{array}{ll}
%%     1278\frac{|\Delta|^{1/2}}{\log|\Delta|}\eps, &:\text{if
%%     }3\cdot 10^{10}\le |\Delta|\le 
%%     10^{15},  \\
%%     17806|\Delta|^{1/4} \eps &:\text{if }1.3\cdot 10^{10} \le |\Delta|< 3\cdot
%%     10^{10}, \\
%%     15545|\Delta|^{1/4} \eps &:\text{if }10^{10}\le |\Delta|< 1.3\cdot 10^{10}.     
%%   \end{array}\right.
%% \end{equation*}
\end{proposition}

\begin{proof}
As in the proof of Theorem~\ref{thceps}, we
set
$$
I=\left(\frac{|\Delta|^{1/2}}{\sqrt3+2\eps}, \frac{|\Delta|^{1/2}}{\sqrt3}\right]. 
$$
We again want to count pairs of integers $(a,b)$ such that 
\begin{align*}
%\label{eainibin}
a\in I,\qquad b\in [-a,-a(1-2\eps))\cup (a(1-2\eps),a], \qquad b^2\equiv \Delta\bmod a. 
\end{align*}
Lemmas~\ref{lrootsmoda} and~\ref{ltrivi} 
imply that,  for
every fixed~$a$, the number of  suitable~$b$ 
does not exceed ${(4\eps\gcd_2(a,\Delta)+2) 2^{\omega(a/\gcd(a,\Delta))+1}}$.
%Using Lemma~\ref{ltrivi}, we obtain 
Hence
%; moreover, if ${8\nmid a}$ then $2^{\omega(a)+1}$ can be replaced by ${2^{\omega(a)}}$. Hence 
\begin{equation}
\label{ethreesums}
\CC_\eps(\Delta) \le 8\eps\sum_{a\in I\cap\Z} {\gcd}_2(a,\Delta)
2^{\omega(a/\gcd_2(a,\Delta)^2)}
+ 4\sum_{a\in I\cap\Z}2^{\omega(a)}%.%+2\sum_{a\in I\cap8\Z}2^{\omega(a)}.
\end{equation}
where we used that $\gcd_2(a,\Delta)^2$ divides $\gcd(a,\Delta)$. 
We estimate each of the terms separately.

\subsubsection*{Estimating the first term in~\eqref{ethreesums}}
Recall that $\Delta / \tilf^2$ is a square-free integer. We have 
\begin{equation}
\label{elefthand}
\sum_{a\in I\cap\Z} {\gcd}_2(a,\Delta) 2^{\omega(a/\gcd_2(a,\Delta)^2)} \le \sum_{d\mid \tilf}d\sum_{a\in I\cap d^2\Z}  2^{\omega(a/d^2)}.
\end{equation}
To estimate the inner sum, write ${a\in I\cap d^2\Z}$ as ${a=d^2a'}$ with 
$$
a'\in d^{-2}I = \left(d^{-2}\frac{|\Delta|^{1/2}}{\sqrt3+2\eps}, d^{-2}\frac{|\Delta|^{1/2}}{\sqrt3}\right].
$$
%% Since ${\omega(a) \le \omega(d)+\omega(a')}$, we obtain 
%% $$
%% \sum_{a\in I\cap d^2\Z}  2^{\omega(a)} \le 2^{\omega(d)}\sum_{a'\in d^{-2}I\cap \Z}  2^{\omega(a')}.
%% $$
We estimate $\sum_{a'\in d^{-2}I\cap\Z} 2^{\omega(a')}$  using~\eqref{esqrtab} with
$$
A=d^{-2}\frac{|\Delta|^{1/2}}{\sqrt3+2\eps}, \quad B=d^{-2}\frac{|\Delta|^{1/2}}{\sqrt3}. 
$$ 
Since ${|\Delta|\le  10^{15}}$, we have ${B\le 2\cdot 10^7}$. From 
$$
B-A\le \frac23 d^{-2}|\Delta|^{1/2}\eps, \quad \log B \le \frac12\log |\Delta|-\frac12\log 3,
$$
we obtain 
$$
\sum_{a'\in d^{-2}I\cap \Z}  2^{\omega(a')}\le \frac{|\Delta|^{1/2}}{3d^2}(\lambda_0 \log |\Delta| +2\lambda_0+2\lambda_1-\lambda_0\log3)\eps+\frac{1.722|\Delta|^{1/4}}{\sqrt[4]{3}d}  
$$
as long as $B\ge 1$. If $B<1$, then the sum on the left is $0$ and the
inequality remains valid as the right side is clearly positive. 
Hence the left-hand side of~\eqref{elefthand} is bounded by 
\begin{align*}
%\sum_{a\in I\cap\Z} {\gcd}_2(a,\Delta) 2^{\omega(a/\gcd_2(a,\Delta)^2)} &\le 
\frac{|\Delta|^{1/2}}{3}(\lambda_0 \log |\Delta|+ 2\lambda_0+2\lambda_1-\lambda_0\log3)\eps\frac{\sigma_1(\tilf)}{\tilf}%\\&\hphantom{\le}
+\frac{1.722|\Delta|^{1/4}}{\sqrt[4]{3}}\sigma_0(\tilf),%\sum_{d\mid \tilf}2^{\omega(d)},
\end{align*}
where we use  notation~\eqref{eq:standardnotation} and the identity
$\sigma_1(\tilf)/\tilf = \sum_{d \mid \tilf} d^{-1}$. %~\eqref{evarsigma}.  

Recall that $\tilf \le |\Delta|^{1/2}\le 3.2\cdot 10^7$.
Hence Lemma \ref{lem:varsigma01} implies that
$$
\sigma_1(\tilf)/\tilf \le %\sigma_1(21621600)/21621600= 
3472/715, \qquad \sigma_0(\tilf)\le 8.5 \tilf^{1/4} \le 8.5|\Delta|^{1/8}. 
$$
Taking into account the factor $8\eps$, the first term in~\eqref{ethreesums} is thus at most
\begin{alignat*}1
   |\Delta|^{1/2} &(8\log|\Delta| + 28)\eps^2 +
   89{|\Delta|^{3/8}}\eps.
%       158\frac{|\Delta|^{1/2}}{\log|\Delta|}\eps.
\end{alignat*}

\subsubsection*{Estimating the second  term in~\eqref{ethreesums}}
Set
$$
A=\frac{|\Delta|^{1/2}}{\sqrt3+2\eps}, \qquad B=\frac{|\Delta|^{1/2}}{\sqrt3}.
$$ 
From~\eqref{emidrange} and ${0<\eps \le 1/3}$ we deduce ${A\ge 4\cdot10^4}$  and ${B\le
2\cdot 10^7}$. This allows us to apply~\eqref{esqrtlogab}, and we obtain
\begin{align*}
\sum_{a\in I\cap\Z}2^{\omega(a)} &\le \frac{|\Delta|^{1/2}}{3}(\lambda_0 \log |\Delta|+2\lambda_0+2\lambda_1 -\lambda_0\log3)\eps\\
&\hphantom{\le}+\frac{4.865|\Delta|^{1/4}}{\sqrt[4]{3}\log(|\Delta/3|^{1/2})}\\
 &\le \frac{|\Delta|^{1/2}}{3}(\lambda_0 \log
|\Delta|+2\lambda_0+2\lambda_1 -\lambda_0\log3)\eps\\
&\hphantom{\le}+\frac{2\cdot
  4.865|\Delta|^{1/4}}{\sqrt[4]{3}\log |\Delta|}
\frac{\log10^{10}}{\log(10^{10}/3)}.
%% &\le \frac{|\Delta|^{1/2}\eps}{3}(\lambda_0 \log |\Delta|+2\lambda_0+2\lambda_1 -\lambda_0\log3)+\frac{2\cdot7.532}{\sqrt[4]{3}}\cdot\frac{\log(10^8)}{\log(10^8/3)}\cdot\frac{|\Delta|^{1/4}}{\log|\Delta|},
\end{align*}
where  for the last estimate we used the assumption
 ${|\Delta|\ge
  10^{10}}$. 
Minding the factor~$4$ we find that 
the second term in~\eqref{ethreesums} is at most
\begin{equation*}
  |\Delta|^{1/2} (0.811 \log|\Delta| + 2.829 )\eps  + 31.06
   \frac{|\Delta|^{1/4}}{\log |\Delta|}. 
\end{equation*}
This concludes our proof of Proposition \ref{prop:Cepsmidrange}.
\end{proof}

\subsection{Proof of Theorem~\ref{thmidrange}}
\label{ssproofmidrange}
Suppose that~$\alpha$ is a singular unit of discriminant~$\Delta$
and set ${X=|\Delta|}$. Assuming  that ${10^{10}\le X< 10^{15}}$,
we arrive at a contradiction. 
%We already know hat ${X<10^{15}}$ from our result in Section
As in   Section~\ref{stenfourteen}  we use the estimates~\eqref{etrivx} and~\eqref{ecolmx} which follow from Propositions~\ref{ptriv}
and~\ref{phlb2}, respectively. 

We may no longer use Corollary~\ref{cupper}, because its hypothesis ${X\ge 10^{14}}$ is not valid in our current range. Instead, we will apply Theorem~\ref{theceps}   directly now. For this, let
${\eps \in (0,4\cdot 10^{-3}]}$. 

{\sloppy

We define~$Y$ as in (\ref{ey}) and recall that ${Y\le 0.01 + \height(\alpha)}$. We find
\begin{equation}
  \label{eq:midrangeYbound}
  Y \le 3 \frac{\CC_\eps(\Delta)}{\CC(\Delta)}
  \log X + 3 \log(\eps^{-1}) - 10.65. 
\end{equation}
Using Proposition \ref{prop:Cepsmidrange} we find
\begin{alignat*}1
  1&\le 3 (8\eps^2 + 0.811\eps)
  \frac{X^{1/2}(\log X)^2}{Y\CC(\Delta)} + 3 {(28\eps^2 + 2.829
    \eps)}\frac{X^{1/2}\log X}{Y\CC(\Delta)}  \\&
\hphantom{\le}    +
3 \frac{89 \eps X^{3/8}\log X}{Y\CC(\Delta)}  
%  3\eps \frac{158 X^{1/2}}{Y\CC(\Delta)}
 + 3\cdot 31.06
  \frac{X^{1/4}}{Y\CC(\Delta)} 
+ \frac{3\log(\eps^{-1}) - 10.65}{Y}.
\end{alignat*}
For all but the final term on the right we use ${Y\CC(\Delta)\ge \pi X^{1/2}}$
and for the remaining term we use
${Y\ge \frac{3}{\sqrt{5}} \log X - 9.78 > 0}$, as ${X\ge  10^{5}}$,  
to get
\begin{equation}
  \label{eq:1lowerbound}\begin{aligned}
  1&\le \frac{3}{\pi}(8\eps^2 + 0.811\eps) (\log X)^2
  +\frac{3}{\pi}(28\eps^2 + 2.829\eps) \log X\\
  &\hphantom{\le}+  \frac{267}{\pi}\eps\frac{\log X}{X^{1/8}}  %\frac{\log X}{X^{1/8}}\\
  +\frac{93.18}{\pi X^{1/4}}
   +
    \frac{3\log(\eps^{-1})-10.65}{\frac{3}{\sqrt 5}\log X - 9.78}.
\end{aligned}
\end{equation}
Our choice is ${\eps = 
10^{-4}}$.
The first two terms in the right-hand side of~\eqref{eq:1lowerbound} are 
monotonously increasing, and the remaining three terms are  decreasing for ${X\in[10^{10}, 10^{15})}$; note that ${x\mapsto(\log x)/x^{1/8}}$ is
decreasing for ${x \ge 3000 > e^8}$. Using ${X < 10^{15}}$  for the first two terms and ${X\ge  2\cdot 10^{10}}$ for the
 remaining three terms, we see that the right-hand side of~\eqref{eq:1lowerbound} is strictly smaller than $0.962$ if
${X \in [2\cdot10^{10}, 10^{15})}$. Similarly, we infer that it is strictly smaller than $0.960$ if ${X \in [10^{10}, 2\cdot10^{10})}$. This completes the proof of Theorem~\ref{thmidrange}. 
\qed

}

\section{Handling the low-range \texorpdfstring{$3\cdot 10^{5}\le |\Delta|< 10^{10}$}{3.10to5<|Delta|<10to10}}
\label{slowrange}
We now deal with the low-range ${|\Delta|\in [3\cdot10^5, 10^{10})}$. For this range the upper
bound on $\CC_\eps(\Delta)$ arises from a computer-assisted search algorithm.

We prove the following. 
\begin{theorem}
  \label{thlowrange}
  Let~$\Delta$  be the discriminant of a singular unit. Then
 ${|\Delta|\notin [3\cdot10^5, 10^{10})}$. 
\end{theorem}

The proof relies on the following lemma. 

\begin{lemma}
  \label{lem:lowrangeCebound}
  Let $\Delta$ be the discriminant of a singular modulus.
  \begin{enumerate}[label={(\roman*)}]
  \item 
  \label{itenten} If ${10^7\le|\Delta|<  10^{10}}$ and ${\eps = 10^{-3}}$, then
    ${\CC_\eps(\Delta)\le 16}$.
      \item 
      \label{itenseven}
      If ${3\cdot10^5\le|\Delta|< 10^7}$ and ${\eps = 4\cdot 10^{-3}}$, then
    ${\CC_\eps(\Delta)\le 6}$. 
  \end{enumerate}
\end{lemma}
\begin{proof}
Let $X_{\min}$ and $X_{\max}$ be positive integers satisfying ${ X_{\min}<X_{\max}}$, and let ${\eps \in (0, 1/3]}$. We want to bound ${\CC_\eps(\Delta)}$ for all~$\Delta$ in the interval ${[-X_{\max},-X_{\min}]}$. 
 
Recall that $\CC_\eps(\Delta)$ counts the triples    $(a,b,c)$ satisfying~\eqref{ekuh} such that 
$$
\tau=\tau(a,b,c)= \frac{b + \sqrt{\Delta}}{2a}
$$
satisfies
  \begin{equation}
    \label{eq:zrestriction}
    \min\{|\tau-\zeta_3|,|\tau-\zeta_6|\} < \eps. 
  \end{equation}
Lemmas~\ref{lgauss}\ref{iadelta} and~\ref{lshort} imply that such triples satisfy 
\begin{align*}
\frac{|\Delta|^{1/2}}2 & \le \hphantom{|}c \hphantom{|}\le \frac{|\Delta|^{1/2}}{\sqrt3}(1+\sqrt3\eps+\eps^2),\\
c/(1+\sqrt3\eps+\eps^2)&<\hphantom{|}a\hphantom{|}\le c,\\
a(1-2\eps)&<|b|\le a. 
\end{align*}
Note that, since ${\eps \in (0, 1/3]}$, we have ${b\ne 0}$ and 
${(1+\sqrt3\eps+\eps^2)/\sqrt3<1}$. 

Hence, to bound 
$\CC_\eps(\Delta)$ on the interval ${[-X_{\max},-X_{\min}]}$,
it suffices, for every~$\Delta$ in this interval, to count the triples $(a,b,c)$ satisfying 
\begin{align*}
{X_{\min}^{1/2}}/2 & \le c \le X_{\max}^{1/2},\\
c/(1+\sqrt3\eps+\eps^2)&<a\le c,\\
a(1-2\eps)&<b\le a 
\end{align*}
and ${b^2-4ac=\Delta}$, and multiply the maximal count by~$2$ (because we counted only triples with positive~$b$). 
We phrase this counting procedure formally as   Algorithm~\ref{algo:CCDelta}.

For a correct implementation, we have to avoid floating point arithmetic in determining
the upper bounds on~$a$ and~$b$ used in the inner two for-loops. For this, we note that ${\eps = 10^{-3}}$
implies that 
$$
0.998 c \le c/(1+\sqrt3\eps+\eps^2), \qquad 0.998a=a(1-2\eps). 
$$
Similarly, ${\eps = 4\cdot 10^{-3}}$
implies that 
$$
0.993 c \le c/(1+\sqrt3\eps+\eps^2), \qquad 0.992a=a(1-2\eps). 
$$
As we are only interested in an upper bound on $\CC_\eps(\Delta)$ for these two specific values of~$\eps$, we use these
weaker rational bounds in our implementation of
Algorithm~\ref{algo:CCDelta}  by means of a
\textsf{C}-program\footnote{A link to our program
  \textsf{algorithm1.c} is on the
second-named author's homepage. The running time on a regular desktop
(Intel Xeon CPU E5-1620 v3, 3.50GHz, 32GB RAM) was under a minute for item~\ref{itenten} and a few milliseconds for item~\ref{itenseven}. Its memory
usage for ${[X_{\min},X_{\max}] = [1, 10^{10}]}$ is significant (5 GB) but this can be overcome by splitting
${[1, 10^{10}]}$ into subintervals and running the program separately for each interval. This feature
is also implemented in our program through the macro \textsf{DISC\underline{\ }BLOCK\underline{\ }SIZE}.}. It
verifies directly the assertions of the lemma.
\end{proof}

    \begin{algorithm}[t]{\footnotesize
      \SetKwInOut{Input}{Input}\SetKwInOut{Output}{Output}
      
    \Input{Two positive integers $X_{\mathrm{min}} <
      X_{\mathrm{max}}$ and $\eps > 0$}

    \Output{an upper bound for $\CC_\eps(\Delta)$ for all
      discriminants $\Delta\in [-X_{\mathrm{max}},-X_{\mathrm{min}}]$}
    \BlankLine

    $counter \leftarrow$ pointer to array of length
    $X_{\mathrm{max}}-X_{\mathrm{min}}+1$ initialized to $0$\;

    $bound \leftarrow 0$\;
    
    \For{$c\leftarrow \lfloor X_{\mathrm{min}}^{1/2}/2\rfloor$ \KwTo
      $\lfloor X_{\mathrm{max}}^{1/2}\rfloor$}{
      \For{$a\leftarrow \lfloor c/(1+\sqrt 3 \eps+\eps^2) \rfloor$ to $c$}{
        \For{$b\leftarrow \lfloor (1-2\eps)a\rfloor$ to $a$}{
          $X\leftarrow 4ac-b^2$\;
          \If{$X\ge X_{\mathrm{min}}$  and $X \le X_{\mathrm{max}}$
          }{$pos \leftarrow X-X_{\mathrm{min}}$\;
            $counter[pos] \leftarrow counter[pos]+2$\;
            \lIf{$counter[pos] > bound$}{$bound\leftarrow counter[pos]$}
            }
        }
    }}
    \Return{$bound$}\;
    \caption{Compute an upper bound for $\CC_{\epsilon}(\Delta)$ in the
      range ${\Delta\in [-X_{\mathrm{max}},-X_{\mathrm{min}}]}$}\label{algo:CCDelta}
      }%
    \end{algorithm}

%\end{proof}

\begin{proof}[Proof of Theorem~\ref{thlowrange}]
  Assume that~$\alpha$ is a singular unit of discriminant ${\Delta\in (-10^{-10}, -3\cdot10^5]}$.
%  By Theorem~\ref{thmidrange} we have ${X=|\Delta|<10^{10}}$. 
Let ${0<\eps\le 4\cdot 10^{-3}}$. We set again~$Y$  as in (\ref{ey}).
  As in the proof of Theorem~\ref{thmidrange} we find (\ref{eq:midrangeYbound}).  
  %% Using Theorem \ref{theceps} with 
  %% (\ref{etrivx}) and (\ref{ecolmx}) gives
We infer that
  \begin{alignat*}1
    1 &\le 3\frac{\CC_\eps(\Delta)}{\CC(\Delta) Y} \log X +
    \frac{3\log(\eps^{-1}) - 10.65}{Y}\\
    & \le \frac {3\CC_\eps(\Delta)}{\pi} \frac{\log X}{X^{1/2}} +
    \frac{3\log(\eps^{-1}) - 10.65}{\frac{3}{\sqrt{5}}\log X - 9.78}
  \end{alignat*}
where we use ${\CC(\Delta)Y \ge \pi X^{1/2}}$
  and ${Y\ge \frac{3}{\sqrt{5}} \log X - 9.78}$. 

  If ${X\in [10^7,10^{10})}$ then we set ${\eps = 10^{-3}}$ and use the estimate
  ${\CC_{\eps}(\Delta) \le 16}$ from Lemma~\ref{lem:lowrangeCebound}\ref{itenten}. 
  Recall that ${x\mapsto (\log x)/x^{1/2}}$
  is decreasing for ${x\ge e^2}$. So we find
  \begin{equation*}
    1 \le \frac{3\cdot 16}{\pi}\frac{\log(10^7)}{10^{7/2}} +
    \frac{3\log(1000) - 10.65}{\frac{3}{\sqrt{5}} \log(10^7)- 9.78}< 0.929,
  \end{equation*}
  a contradiction.

When ${X\in [3\cdot10^5,10^7)}$  we set ${\eps = 4\cdot10^3}$.  Then
  ${\CC_{\eps}(\Delta)\le 6}$ by  Lem\-ma~\ref{lem:lowrangeCebound}\ref{itenseven}.
Using 
   ${X\ge 3\cdot 10^5}$ we find as before 
  \begin{equation*}
    1 \le \frac{3\cdot 6}{\pi}\frac{\log(3\cdot 10^5)}{(3\cdot 10^{5})^{1/2}} +
    \frac{3\log(250) - 10.65}{\frac{3}{\sqrt{5}} \log(3\cdot 10^{5})- 9.78}< 0.961,
  \end{equation*}
  another contradiction which completes this proof. 
\end{proof}

%\section{The final attack}

\section{The extra low-range}

\label{sfinal}

The results of the three previous sections reduce the proof of Theorem~\ref{thmain}  to the
following assertion.

\begin{theorem}
\label{thm:habeggerspari}
    Let~$\Delta$ be the discriminant of a singular unit. Then 
  ${|\Delta|\ge 3\cdot 10^5}$. 
\end{theorem}
\begin{proof}
  Let $\alpha$ be a singular unit of discriminant $\Delta$.  We write
  ${X=|\Delta|}$. 
We may assume that ${X\ge 4}$ because the only singular modulus of discriminant $-3$ is
${j(\zeta_3)=0}$, 
which is  not an algebraic unit.

Recall from Section~\ref{sceps} that the Galois conjugates of~$\alpha$ are precisely the singular moduli
$j(\tau)$, where ${\tau = \tau(a,b,c)}$  with $(a,b,c)$ as in~\eqref{ekuh}. The imaginary part of  such~$\tau$ is
$X^{1/2}/(2a)$ and ${a\le (X/3)^{1/2}}$
by Lemma~\ref{lgauss}\ref{iadelta}. 
  Lemma~\ref{lttsn}  implies that
$$
|j(\tau)|\ge e^{2\pi X^{1/2}/(2a)} -2079
  =e^{\pi X^{1/2}/a}-2079 > 23^{X^{1/2}/a}-2079
$$
  as ${e^\pi > 23}$. Using Lemmas~\ref{lanal} and~\ref{lliouv}, we find that
$$
|j(\tau)| \ge 42700 \min \left\{ \frac{\sqrt3}{4X},4\cdot 10^{-3}\right\}^3.
$$
These
  bounds together show that
  \begin{equation}
\label{eq:finaljlb}
    |j(\tau)|\ge \max\left\{ 23^{\lfloor X^{1/2}/a\rfloor}-2079,
    42700\min\Bigl\{\frac2{5X},\frac1{250}\Bigr\}^3\right\}. 
  \end{equation}
Based on this observation,   Algorithm~\ref{algo:exclude} prints a list of discriminants of potential singular
units in the range ${[-X_{\max},-4]}$. For this purpose, it computes a 
rational lower bound~$P$ for the absolute value of the
$\Q(\alpha)/\Q$-norm  of each singular moduli in this range. Those singular
moduli where ${P \le 1}$ are then flagged as potential singular units.

 We have implemented this algorithm as a \textsf{PARI}
 script\footnote{A link to our \textsf{PARI} script  \textsf{algorithm2.gp}
is on  the second-named author's homepage. The running time is about 23 minutes  on a regular desktop computer (Intel Xeon CPU E5-1620 v3, 3.50GHz, 32GB RAM). The only floating point operation used  approximates $X^{1/2}$
  which leads to  ${n=\lfloor X^{1/2}/a\rfloor}$.
To rule out a rounding error in the floating point arithmetic we
compare $(an)^2$  with~$X$ in our implementation.}. 
The script flags only $-4$, $-7$ and $-8$ as discriminants of potential singular units. The
singular moduli of these discriminants are well-known \cite[(12.20)]{Cox}: they are $12^3$, $-15^3$ and $20^3$, respectively. None of them is a unit, which concludes the proof. 
\end{proof} 
  
  \begin{algorithm}[t]{\footnotesize
          \SetKwInOut{Input}{Input}\SetKwInOut{Output}{Output}
          \Input{An integer $X_{\mathrm{max}}\ge 1$}
          \Output{Print a list containing all discriminants
            in $[-X_{\mathrm{max}},-4]$  that are attached to a potential
            singular unit.}

          \For{$X\leftarrow 4$ \KwTo $X_{\mathrm{max}}$}{
              $\Delta \leftarrow -X$\;
            
            \lIf{$\Delta\equiv 2 \text{ or }3 \mod 4$}{
              next $X$
             }

$P\leftarrow 1$\;
            \For{$a\leftarrow 1$ \KwTo $\lfloor \sqrt{X/3}\rfloor$} {
 %             \For{$(a,b,c)\in T_{\Delta}$}{
              $n\leftarrow \lfloor X^{1/2}/a\rfloor$\;
              \For{$b\leftarrow -a+1$ \KwTo $a$}{
                \lIf{$b^2  \not\equiv \Delta \mod 4a$}{
                  next $b$
                }
                $c\leftarrow (b^2-\Delta)/(4a)$\;
                \lIf{$a>c$}{next $b$}
                \lIf{$a=c \text{ and } b<0$}{next $b$}
                \lIf{$\mathrm{gcd}(a,b,c)\not=1$}{next $b$}
                %% \If{$(an)^2 > X$}{
                %%   print FAIL\;
                %%   exit\;
                %% }
  $P\leftarrow P \cdot \max\{23^{n
        }-2079,42700\min\{2/(5X),1/250\}^3\}$\;
}}

\lIf{$P\le 1$}{print $\Delta$}}
\caption{Exclude singular units}          \label{algo:exclude}
}%
  \end{algorithm}

As indicated in the introduction, Theorem~\ref{thmain} is the combination of Theorems~\ref{thhighrange},~\ref{thmidrange},~\ref{thlowrange} and~\ref{thm:habeggerspari}.

{\footnotesize

\def\cprime{$'$}
\providecommand{\bysame}{\leavevmode\hbox to3em{\hrulefill}\thinspace}
\providecommand{\MR}{\relax\ifhmode\unskip\space\fi MR }
\providecommand{\MRhref}[2]{%
  \href{http://www.ams.org/mathscinet-getitem?mr=#1}{#2}
}
\providecommand{\href}[2]{#2}

%% \bibliographystyle{amsplain}
%% \bibliography{literature}

\bigskip

Yuri Bilu,  \textsc{IMB, Université de Bordeaux and CNRS,
351 cours de la Libération,
33405 Talence cedex, France}

{\tt yuri@math.u-bordeaux.fr}

\medskip

Philipp Habegger,
\textsc{Department of Mathematics and Computer Science,
University of Basel, Spiegelgasse 1,
4051 Basel,
Switzerland}

{\tt philipp.habegger@unibas.ch}

\medskip
Lars K\"uhne,
\textsc{Department of Mathematics and Computer Science,
University of Basel,
Spiegelgasse 1,
4051 Basel,
Switzerland}

{\tt lars.kuehne@unibas.ch}

%% \vfill
%% \medskip
%% \noindent {\tt \jobname.tex} created  on September  11, 2016. Latest  update on \today.  \hfill

}

\

\end{document}